\documentclass[11pt]{article}

\usepackage{amsmath,amssymb,amsfonts}
\usepackage{algorithm}
\usepackage{algorithmic}
\usepackage{cite}
\usepackage{comment}
\usepackage{graphicx}
\usepackage{textcomp}
\usepackage{xcolor}
\usepackage{subfigure}
\usepackage{bm}
\usepackage{amsthm}
\usepackage{booktabs}
\usepackage{url}
\usepackage[authoryear, square]{natbib}
\usepackage[margin=1in]{geometry}
\usepackage{tabularx}
\usepackage{hyperref}

\theoremstyle{definition}
\newtheorem{theorem}{\textbf{Theorem}}

\newtheorem{proposition}{\textbf{Proposition}}
\newtheorem{remark}{\textbf{Remark}}

\newtheorem{condition}{\textbf{Condition}}

\newtheorem{corollary}{\textbf{Corollary}}

\begin{document}

\title{Voluntary Renewable Programs: Optimal Pricing and Revenue Allocation}
\author{
    Zhiyuan Fan\thanks{Department of Earth and Environmental Engineering, Columbia University, USA. Correspondence. Email: zf2198@columbia.edu} \and
    Tianyi Lin\thanks{Department of Industrial Engineering and Operations Research, Columbia University, USA.  Email: tl3335@columbia.edu} \and
    Bolun Xu\thanks{Department of Earth and Environmental Engineering, Columbia University, USA. Email: bx2177@columbia.edu}
}
\date{\today}

\maketitle

\begin{abstract}
This paper develops a multi-period optimization framework to design a voluntary renewable program (VRP) for an electric utility company, aiming to maximize total renewable energy deployments. In the business model of VRP, the utility must ensure it generates renewable energy up to the total amount of contract during each market episode (i.e., a year), while all the revenue collected from the VRP must either be used to invest in procuring renewable capacities or to maintain the current renewable fleet and infrastructure. We thus formulate the problem as an optimal pricing problem coupled with revenue allocation and renewable deployment decisions. We model the demand function of voluntary renewable contracts as an exponential decay function based on survey data.
We analytically derive the optimal pricing policy of the VRP as a function of the current grid carbon intensity. We prove that a myopic policy is conditionally optimal, which maximizes renewable capacity in each period, attains the long-run optimum due to the utility's revenue-neutral constraint. We show different binding conditions and marginal values of decision variables correspond to different phases of the energy transition, and that the utility should strategically design its revenue-sharing decisions, balancing investments in renewable expansion and subsidizing existing renewable fleets.
Finally, we show that voluntary renewable programs can only extend renewable penetration but cannot achieve net-zero emissions or a fully renewable grid. This pricing–allocation–expansion framework highlights both the potential and limitations of voluntary renewable demand, providing analytical insight into optimal policy design and the qualitative shifts occurring during the energy transition process.
\end{abstract}

\section{Introduction}

The decarbonization of electricity systems requires the large-scale deployment of renewable generation such as wind and solar. Rapid cost declines have made these technologies increasingly competitive with conventional fossil-fuel generation, enabling substantial commercial deployment in many regions~\cite{irena2023cost}. At the same time, integrating high shares of renewables presents both operational and economic challenges for power systems. Because renewable resources are inherently variable and intermittent, maintaining system reliability requires complementary investments in flexibility resources such as energy storage, flexible gas turbines, and advanced grid management~\cite{denholm2015duck}. 

In addition to these operational challenges, renewable generation interacts in complex ways with the design of competitive electricity markets. Renewable generators typically have near-zero marginal operating costs~\cite{borenstein_private_2012,joskow_challenges_2019}, while wholesale electricity markets commonly rely on marginal pricing mechanisms in which all generators receive the market-clearing price determined by the marginal unit. As renewable penetration increases, this merit-order effect tends to suppress wholesale electricity prices during periods of high renewable output, reducing the market value of renewable generation~\cite{hirth2013market,brown2018price}. Consequently, even when renewable technologies are cost-competitive on a levelized basis, the revenue they earn from electricity markets can decline as their deployment expands. This dynamic implies that competitive electricity markets alone may not generate sufficient investment incentives to sustain large-scale renewable expansion.

\subsection{Voluntary Demand for Renewable Electricity}
In parallel with these market limitations, voluntary demand for renewable electricity has grown rapidly in recent years and is increasingly emerging as an important driver of renewable deployment. Large corporations—including technology firms such as Google, Apple, and Microsoft—have committed to sourcing renewable electricity for their operations through long-term procurement contracts, while manufacturers and other commercial consumers increasingly purchase renewable energy to reduce their carbon footprints~\cite{oshaughnessy2021corporate,trivella2023corporate}. At the same time, many utilities now offer voluntary renewable electricity programs that allow customers to purchase electricity backed by renewable generation at a price premium. Table~\ref{tab:green_power_markup} summarizes representative examples of such programs in the United States, showing that full renewable electricity plans typically impose a price premium ranging from approximately 0.75 to 2.2 cents per kilowatt-hour.

In practice, voluntary renewable demand is implemented through several institutional mechanisms.
Renewable energy certificates (RECs) a widely used mechanisms for monetizing the environmental attributes of renewable electricity. A REC represents a tradable claim on the renewable generation associated with one unit of electricity, allowing buyers to support renewable production independently of the underlying electricity supply~\cite{morthorst_development_2000}. Because RECs can be traded separately from electricity, they provide an additional revenue stream for renewable generators and create a market-based instrument for supporting renewable deployment. In compliance markets, REC prices are disciplined by regulatory obligations such as renewable portfolio standards, and a substantial literature studies their price dynamics and interactions with electricity markets~\citep{jensen_interactions_2002,tanaka_market_2013,shrivats_meanfield_2022,hulshof_performance_2019}. However, voluntary REC markets often operate with thin liquidity, bilateral trading, and limited transparency, which can weaken price discovery and reduce the credibility of the resulting price signals for renewable investment~\cite{hulshof_performance_2019,frei_liquidity_2018,wimmers_european_2023}.

Long-term power purchase agreements (PPAs) represent a second major mechanism through which renewable attributes are monetized. Under a PPA, large electricity consumers contract directly with renewable generators through long-term agreements that provide stable revenue streams and support project financing. Corporate PPAs have become an important channel for voluntary renewable procurement, particularly for large technology firms and other energy-intensive industries~\citep{trivella_meeting_2023}. While PPA can improve revenue stability and facilitate renewable investment, they also bundle electricity supply, financial risk allocation, and environmental attributes into a single contract. This bundling can obscure the standalone value of renewable attributes and limit transparent price formation for renewable credits~\cite{backstrom_corporate_2024,ghiassi-farrokhfal_making_2021}. 

Another important limitation of corporate renewable procurement through PPAs is that the contract price primarily reflects the generation cost of renewable electricity but does not incorporate the broader system costs associated with delivering and integrating that electricity into the power grid. In particular, the costs of transmission expansion, distribution upgrades, and system balancing resources—such as energy storage, flexible generation, and reserve capacity—are typically borne elsewhere in the electricity system rather than reflected in the PPA price. Empirical market data suggest that recent corporate renewable PPAs in North America have prices on the order of \$50--70/MWh, corresponding to roughly \$0.05--0.07/kWh for utility-scale wind and solar contracts~\cite{levelten2024ppa}. Because these prices are often close to prevailing wholesale electricity prices in competitive power markets, the effective ``green premium'' for corporate renewable procurement is frequently small and may even be negative in regions with abundant renewable resources~\cite{oshaughnessy2021corporate}. While this price competitiveness has facilitated rapid growth in corporate renewable procurement, it also highlights that PPA prices capture only a portion of the true economic cost of renewable electricity supply. As renewable penetration increases, the omitted costs of delivery, integration, and system balancing become increasingly significant, suggesting that PPA-based procurement alone may not provide a sustainable framework for coordinating large-scale renewable deployment.

As a result, both REC markets and PPA-based procurement may generate fragmented or inconsistent price signals for renewable investment~\cite{ranson_linkage_2016}. 

\subsection{Voluntary Renewable Programs}
Voluntary Renewable Programs (VRPs) represent an alternative mechanism that more directly connects renewable procurement with the underlying costs of electricity delivery and system integration. Under VRPs, utilities offer customers the option to purchase electricity backed by renewable generation at a specified price premium. Unlike RECs and PPAs—which are typically traded in wholesale certificate markets or negotiated through bilateral corporate procurement contracts—VRPs operate at the retail level and aggregate consumer willingness to pay through the utility. Because utilities remain responsible for electricity delivery, retail pricing, and investments in the distribution infrastructure required to accommodate growing renewable penetration, they are uniquely positioned to incorporate both generation costs and system-level integration costs into renewable pricing decisions~\cite{pollitt_lessons_2012}.

\begin{table}[htbp]
\centering
\caption{Survey of voluntary renewable electricity programs and customer price premiums.}
\label{tab:green_power_markup}

\renewcommand{\arraystretch}{1.35}

\begin{tabular}{p{4.5cm} c p{6.5cm}}
\hline
Provider / Program & Premium (USD/kWh) & Notes \\
\hline

Austin Energy \newline
GreenChoice \newline
(Residential) \newline \cite{austin_greenchoice}
& 0.0075 
& Simple fixed premium added to the retail electricity rate for customers selecting a 100\% renewable supply option. \\

Dominion Energy \newline
Green Power \newline
(100\% Option) \newline \cite{dominion_greenpower}
& 0.012 
& Example program charge equivalent to \$12 per month for a household consuming 1000 kWh. \\

Dominion Energy \newline
Rider TRG \newline
(100\% Renewable) \newline \cite{dominion_trg}
& 0.0136 
& Premium component embedded within a tariff rider structure rather than a standalone surcharge. \\

Georgia Power \newline
Simple Solar \newline \cite{georgia_simple_solar}
& 0.0125 
& Solar REC matching program where customers pay a fixed premium to support renewable generation. \\

Xcel Energy \newline
Renewable Connect Flex \newline \cite{xcel_renewable_connect}
& 0.015 
& Subscription-based renewable procurement program priced at \$1.50 per 100 kWh block of consumption. \\

Portland General Electric \newline
Green Future Choice\newline  \cite{pge_greenfuture}
& 0.0094$^{a}$ 
& Premium applies to the voluntary renewable portion of electricity consumption beyond the renewable share already included in the default utility mix. \\

Natick Community Electricity \newline
(Community Choice Program,
Eversource service territory) \newline
\cite{natick_community_electricity}
& 0.0226
& 100\% renewable option priced at 16.32¢/kWh compared with 14.06¢/kWh for the standard renewable plan, implying a premium of about 2.26¢/kWh. \\


\hline
\end{tabular}

\vspace{0.3cm}
\begin{flushleft}
{\footnotesize
$^{a}$ The Green Future Choice program purchases renewable energy certificates to match the non-mandated portion of customer electricity consumption.

}
\end{flushleft}

\end{table}

VRPs are becoming an increasingly important channel for financing renewable and low-carbon energy outside formal compliance regimes. Recent studies document both their rapid growth and their persistent weaknesses in credibility, price formation, and procurement design~\citep{wetterberg_interplay_2024,kreibich_caught_2021,bjorn_renewable_2022}. These challenges are especially salient because VRP demand now arises through a widening range of procurement channels, from large corporate power purchase agreements~\citep{trivella_meeting_2023,taheri_physical_2025} to community choice aggregation programs that enable collective participation in green power markets~\citep{oshaughnessy_empowered_2019}. As the generation mix shifts toward higher shares of renewable energy, utilities must finance distribution reinforcement, grid modernization, advanced control technologies, and local storage or flexibility resources needed to maintain system reliability~\cite{das_overview_2018}. Existing renewable procurement mechanisms such as RECs and PPAs generally do not provide a transparent framework for recovering these system-level costs~\cite{jenkins_improved_2017}. By contrast, a centralized VRP mechanism administered by utilities can translate voluntary consumer demand into renewable procurement decisions while simultaneously accounting for the full economic cost of renewable electricity supply. In this way, VRPs have the potential to provide scalable, system-wide coverage that aligns consumer willingness to pay with both renewable generation expansion and the infrastructure investments required to sustain a high-renewable power system.

Another key advantage of VRP compared to REC and PPA is that it provides a standardized product and translates consumer demand for clean electricity into credible price signals and sustained renewable deployment. Table \ref{tab:green_power_markup} summarizes a survey of utility voluntary renewable electricity programs and their associated customer price premiums, showing that a full renewable electricity plan imposes a price premium ranging from 0.75 cents to 2.2 cents per kilowatt hour. The variation in such pricing premiums can be grounded in factors including geographical dependent renewable abundance, existing renewable capacity, revenue allocation, and program performance. 

\subsection{Research Gap and Paper Contribution}
A growing empirical literature examines voluntary demand for renewable electricity and consumers' willingness to pay for green power products. Early survey-based studies document positive willingness to pay for renewable electricity among residential consumers~\cite{farhar_willingness_1999}, while subsequent meta-analyses and experimental studies estimate the magnitude and determinants of these premiums across different markets~\cite{sundt_consumers_2014,andor_equity_2018}. More recent work also examines voluntary renewable procurement by corporations and other large electricity consumers as an emerging driver of renewable deployment~\cite{oshaughnessy2021corporate,trivella2023corporate}. 


Yet, the exact systematic interaction among these factors in the pricing decision remains poorly understood.  In particular, we lack analytical frameworks that jointly characterize how voluntary demand translates into renewable credit pricing, how the resulting revenues are allocated within the electricity system, and how these mechanisms influence long-run renewable capacity expansion. This gap is particularly important in utility-operated renewable programs, where pricing decisions must simultaneously reflect consumer willingness to pay, renewable generation costs, and the system-level costs of integrating intermittent resources. This underscores the need for an analytical framework that jointly characterizes renewable-credit pricing for the program, revenue allocation, and capacity expansion within a realistic utility setting. 

Building on these trends, this paper develops a theoretical framework for pricing voluntary renewable programs and allocating the resulting revenue between renewable expansion and system integration costs. We formulate a multi-period optimization model in which a utility sets renewable program prices and allocates revenues to support renewable deployment while satisfying financial feasibility constraints. Our analysis yields closed-form characterizations of optimal pricing policies and revenue-sharing rules, and shows that a simple myopic expansion policy can achieve the long-run optimum. We further demonstrate that VRPs can substantially expand renewable penetration but, by themselves, cannot deliver full decarbonization, highlighting the need for complementary policy instruments. Together, these results provide analytical insights into how voluntary renewable demand can be translated into credible price signals and effective renewable deployment in realistic utility settings.

\section{Model and Problem Formulation}

\subsection{Framework}

We adopt the perspective of an electric power utility that offers VRP options to its customers and uses the resulting program revenue to support investment in renewable generation. We consider a benevolent utility whose objective is to maximize renewable capacity expansion while maintaining strict financial feasibility. Let the initial renewable capacity in the utility’s generation portfolio be denoted by \(Q_0\).

The planning horizon is a finite discrete set of time periods
\(
\mathcal{T}=\{1,2,\dots,n\},
\)
where each period represents one decision stage (e.g., a financial year). Notation used throughout the formulation and in subsequent figures is summarized in Table~\ref{table:notations}.

We first present a baseline formulation in which all financial flows are evaluated on a \emph{single integrated balance sheet}. This formulation corresponds to a vertically integrated benchmark in which the utility internalizes both system-level costs and generation-level net costs. The purpose of this benchmark is to isolate the fundamental pricing and expansion incentives of the program. In Section~\ref{sec:revenue_sharing}, we will relax this assumption and explicitly reintroduce revenue allocation across utility and independent power producers.

We start with formulating the optimal VRP pricing problem as a multi-period optimization problem.
In each period \(t\), given the current renewable capacity \(Q_t\), the utility chooses the VRP price \(p_t\) and capacity addition \(q_t\). Define the per-period VRP revenue based on the demand function \(D\):
\[
R(p_t,Q_t)\;:=\;p_t\,D\!\left(\frac{p_t}{e(Q_t)}\right).
\]
The term \(p_t / e(Q_t)\) converts the electricity premium into an effective abatement price. Since \(p_t\) is measured in \$/kWh and \(e(Q_t)\) in ton-CO\(_2\)/kWh, the ratio has units of \$/ton-CO\(_2\), reflecting that VRP demand is driven by the cost of emissions reduction provided by renewable procurement. Further details on the demand specification are provided in Section~2.2.  

All other costs borne by the vertically integrated utility, excluding renewable investment, are summarized by the \emph{non-investment} per-period cost term
\[
C(Q_t)\;:=\;C_S(Q_t)+C_R(Q_t)-f(Q_t)\pi(Q_t),
\]
so that the total per-period financial requirement is given by \(C(Q_t)+k_t q_t\), where \(k_t\) is the unit expansion cost in period \(t\). The underlying components are summarized in Table~\ref{table:notations}.

The multi-period optimization problem is:
\begin{subequations}\label{eq:opt_problem_multi}
\begin{align}
    \max_{\{p_t,\, q_t\}_{t\in\mathcal T}} \quad 
    & \sum_{t \in \mathcal{T}} q_t \label{eq:opt_problem_multi_obj} \\[4pt]
    \text{s.t.}\quad 
    & D\!\left(\frac{p_t}{e(Q_t)}\right) \le f(Q_t), 
    \qquad &&\forall t \in \mathcal T \label{eq:opt_problem_multi_a}\\[2pt]
    & C(Q_t)+k_t q_t \le \;p_t\,D\!\left(\frac{p_t}{e(Q_t)}\right),
    \qquad &&\forall t \in \mathcal T \label{eq:opt_problem_multi_b}\\[2pt]
    & Q_{t+1} = Q_t + q_t, \qquad Q_1 \text{ given (or } Q_0\text{)}, \qquad 
    &&\forall t \in \mathcal T \setminus \{\max\mathcal T\} \label{eq:opt_problem_multi_c}\\[2pt]
    & q_t \ge 0,\quad p_t \ge 0, \qquad &&\forall t \in \mathcal T. \label{eq:opt_problem_multi_d}
\end{align}
\end{subequations}

The utility channels revenue from VRP customers into additional renewable deployment. Accordingly, its fundamental objective~\eqref{eq:opt_problem_multi_obj} is to maximize cumulative renewable capacity additions over the planning horizon, subject to revenue adequacy in each period. Although written as an inequality to reflect budget feasibility, the financial constraint~\eqref{eq:opt_problem_multi_b} will bind whenever the optimal solution exhibits strictly positive expansion \(q_t>0\).



The optimization problem is subject to the following constraints:
\begin{itemize}
    \item \textbf{Renewable-program feasibility}~\eqref{eq:opt_problem_multi_a}:  
    The total renewable credits sold by the program, \(D\!\left(p_t/e(Q_t)\right)\), must not exceed the amount of delivered renewable electricity, \(f(Q_t)\). This ensures that credit sales are backed by usable renewable generation after accounting for curtailment and saturation effects (Table~\ref{table:notations}).

    \item \textbf{Integrated financial feasibility}~\eqref{eq:opt_problem_multi_b}:  
    Per-period VRP revenue \(R(p_t,Q_t)=p_t\,D\!\left(p_t/e(Q_t)\right)\) must cover the integrated net cost of operating and expanding renewable capacity.
    
    To preserve the economic interpretation used later under separated financial accounts, we further decompose left-hand side into two components:
    (i) the \emph{system/program cost} \(C_S(Q_t)+k_t q_t\), borne by the program operator (e.g., integration, balancing, and administration), and
    (ii) the \emph{generation-side net cost} \(C_R(Q_t)-f(Q_t)\pi(Q_t)\), representing renewable operating costs net of wholesale energy-market revenue.
    This decomposition is not needed for the integrated benchmark, but it becomes essential in Section~\ref{sec:revenue_sharing}, where these components are evaluated on separate balance sheets. Detailed definitions and units are summarized in Table~\ref{table:notations}.

    \item \textbf{Capacity state transition}~\eqref{eq:opt_problem_multi_c}:  
    Renewable capacity evolves according to capacity accumulation, so that cumulative capacity in period \(t+1\) equals the previous level plus new additions. Capacity retirements are not modeled.

    \item \textbf{Feasibility bounds}~\eqref{eq:opt_problem_multi_d}:  
    Decision variables are restricted to nonnegativity: \(q_t \ge 0\) and \(p_t \ge 0\).
\end{itemize}

We restrict the utility’s decisions to be strictly financially feasible in each period, relying solely on voluntary demand for renewables. Under this formulation, neither the utility nor renewable generators incur losses, and no intertemporal borrowing or banking of funds is permitted. This profit-neutral, period-by-period feasibility condition reflects standard regulatory and accounting constraints faced by utilities and provides a transparent benchmark for subsequent analysis.

\begin{table}[h]
\centering
\caption{Model Elements and Notation}
\begin{tabularx}{\linewidth}{l l X}
\midrule

\textbf{Sets and indices} 
    & $\mathcal{T}=\{1,2,\dots,n\}$ 
    & Discrete planning horizon (e.g., years). \\
    & $t\in\mathcal{T}$ 
    & Time period index. \\[3pt]

\midrule
\textbf{Decision variables} 
    & $p_t\ge 0$ 
    & VRP price (low-carbon electricity premium) in period $t$ (\$/MWh). \\
    & $q_t\ge 0$ 
    & Renewable capacity addition in period $t$ (MW). \\[3pt]

\midrule
\textbf{State and transition}
    & $Q_t$ 
    & Cumulative installed renewable capacity at the end of period $t$ (MW). \\
    & $Q_t = Q_{t-1}+q_t$ 
    & Capacity state transition; $Q_0$ is given. \\[3pt]

\midrule
\textbf{Exogenous elements} 
    & $e(Q_t)\ge 0$ 
    & System marginal emissions intensity under renewable capacity $Q_t$ (ton-CO$_2$/MWh). \\
    & $f(Q_t)\ge 0$ 
    & Delivered (usable) renewable electricity as a function of capacity $Q_t$, capturing curtailment/saturation effects (MWh). \\
    & $\pi(Q_t)$ 
    & Effective marginal wholesale value received by renewable generators given $Q_t$ (\$/MWh). \\[3pt]

\midrule
\textbf{Demand}  
    & $D\!\left(\frac{p_t}{e(Q_t)}\right)$ 
    & VRP demand as a function of the effective carbon price $p_t/e(Q_t)$ (\$/ton-CO$_2$). \\
    & $D(x)=M e^{-\epsilon x}$ 
    & Baseline demand specification; $M$ is market size and $\epsilon>0$ is demand sensitivity. \\[3pt]

\midrule
\textbf{Cost elements} 
    & $C_R(Q_t)\ge 0$ 
    & Renewable operating/sustaining cost under capacity $Q_t$. \\
    & $C_S(Q_t)\ge 0$ 
    & System integration/program cost under capacity $Q_t$. \\
    & $k_t\ge 0$ 
    & Unit investment cost of new renewable capacity in period $t$ (\$/MW). \\[3pt]

\midrule
\end{tabularx}
\label{table:notations}
\end{table}

\clearpage

\subsection{VRP Demand for Renewable}

Demand for voluntary renewable programs (VPRs) arises from entities that are not legally obligated to procure renewables but choose to do so for strategic, reputational, or market reasons. Large technology companies such as Microsoft, Amazon, and Google purchase renewable credits to demonstrate environmental leadership and enhance brand value~\cite{egli_contribution_2023}. Certain industrial producers—such as emerging green hydrogen projects~\cite{roper_renewable_2025}—seek renewable procurement to qualify for subsidies, partnerships, or market differentiation. Many utilities also offer renewable power options to residential and commercial customers~\cite{dagher_residential_2017}, while green-building certification programs establish standards that incentivize renewable energy use~\cite{agarwal_role_2024}. Collectively, these voluntary buyers constitute the demand base that sustains the program.

\begin{condition}\label{cond:demand_function}
The demand for renewable follows an exponential form:
\begin{equation}\label{eq:demand_function}
    D\!\left(\frac{p}{e(Q)}\right)= M \, \exp\!\left(-\epsilon \frac{p}{e(Q)}\right),
\end{equation}
where \(M\) denotes the total potential market size at zero premium, \(\epsilon>0\) is a demand sensitivity parameter, and \(e(Q)\) is the grid-average emissions intensity given cumulative renewable capacity \(Q\) (ton-\(\mathrm{CO}_2\)/MWh).
\end{condition}

Importantly, the functional form of demand is assumed to be time-invariant. Any temporal variation arises endogenously through changes in the state variable \(Q\), rather than through explicit time dependence of preferences or market size. This formulation isolates the effect of renewable penetration on the emissions baseline—and hence the effective carbon-price signal faced by buyers—while abstracting from exogenous macroeconomic or social shocks.

The exponential specification in equation~\eqref{eq:demand_function} is supported by empirical evidence showing that voluntary demand for renewable declines in a strongly convex manner with respect to price. For instance, \citet{andor_equity_2018} applied a linear approximation to a decreasing convex demand; \citet{mewton_green_2011} estimated a constant-elasticity (power-law) demand function \(D = a p^{\beta}\) with \(\beta=-0.96\), indicating nearly unitary elasticity; and \citet{farhar_willingness_1999} and \citet{sundt_consumers_2014} reported exponential or semi-log relationships between renewable uptake and price.

The intuition for adopting the exponential form is threefold. First, when the renewable premium is zero \((p=0)\), all potential buyers are willing to participate the program, yielding \(D\!\left(0\right)=M\), the full market size. Second, because demand is voluntary rather than mandated, demand approaches zero as the effective carbon price increases indefinitely, i.e.,
\[
\lim_{p \to \infty} D\!\left(\frac{p}{e(Q)}\right) = 0.
\]
Third, the exponential function captures the empirically observed rapid, convex decline in willingness to pay as prices rise. Together, these properties make the exponential specification both behaviorally intuitive and empirically well supported for modeling VRP demand.

The term \(e(Q)\) in the denominator of equation~\eqref{eq:demand_function} represents the average emissions intensity of the power grid (ton-\(\mathrm{CO}_2\)/MWh) given the existing renewable capacity \(Q\). VRP buyers seek to differentiate their electricity consumption from this grid average. When the grid average reaches zero—corresponding to a fully decarbonized or net-zero system—VRP demand vanishes because no additional emissions benefit remains from purchasing renewable through the program.

Dividing the premium price \(p\) (in \$/MWh) by \(e(Q)\) normalizes the renewable premium into an effective carbon price (in \$/ton-\(\mathrm{CO}_2\)). In this sense, VRP demand is driven by the perceived value of avoided emissions, and the normalization \(p/e(Q)\) translates the renewable premium into the implicit carbon-price signal faced by buyers.

\subsection{Grid Properties with Renewable Penetration}

As the penetration level of renewables increases, the power grid undergoes fundamental structural changes. On the one hand, a higher share of renewable generation lowers the system’s average emissions intensity by displacing fossil generation. On the other hand, growing penetration leads to greater curtailment of renewable output, as exemplified by the well-known “duck curve” phenomenon observed in California~\cite{denholm_overgeneration_2015}. At the same time, the effective energy value received by renewables from the wholesale power market diminishes~\cite{hirth_market_2013}, since renewables frequently bid at zero or negative prices. These general trends motivate the functional properties we impose on the grid as a function of cumulative renewable capacity \(Q\).

\begin{condition}\label{ass:grid_functions}
We assume the following grid functional properties with respect to renewable capacity (or penetration level) \(Q\):
\begin{enumerate}
    \item \textbf{Grid-average emissions intensity \(e(Q)\) is decreasing:}
    \[
    e'(Q) < 0.
    \]

    \item \textbf{Delivered renewable electricity \(f(Q)\) is increasing and concave:}  
    \[
    f(0)=0, \qquad f'(Q) > 0, \qquad f''(Q) < 0.
    \]

    \item \textbf{Effective marginal wholesale value received by renewables \(\pi(Q)\) is decreasing:}  
    \[
    \pi'(Q) < 0.
    \]
\end{enumerate}
\end{condition}

The first two assumptions in Condition~\ref{ass:grid_functions} reflect the role of curtailment as renewable penetration increases. Expanding renewable capacity displaces fossil generation and therefore reduces average grid emissions intensity, justifying \(e'(Q) < 0\). At the same time, increasing penetration leads to growing curtailment, so that the amount of renewable electricity that can be effectively delivered increases at a diminishing rate, reflected by \(f'(Q) > 0\) and \(f''(Q) < 0\).

The monotonicity of the emissions-intensity function and the concavity of the delivery function capture saturation effects in both emissions reductions and usable renewable output. Consider solar photovoltaic generation as an example. Solar output is limited to daylight hours, with no production at night. As additional solar capacity is installed, there is a maximum fraction of total load that can be served without storage. The marginal usable output from renewables saturates as curtailment grows, motivating concavity in \(f(Q)\). While grid investments such as storage and transmission can mitigate these effects, they are not expected to reverse the overall trends of increasing curtailment and diminishing marginal impact at higher penetration levels.

The decreasing trend of the effective marginal price \(\pi(Q)\) is less immediate but equally important. This function does not represent the system-wide average wholesale price. Instead, it reflects the average price received by renewable generators during periods when they produce. Under the duck curve phenomenon, periods of high renewable output—such as midday hours for solar—are associated with suppressed wholesale prices. As renewable penetration increases, these low-price periods become more frequent and pronounced, reducing the average marginal revenue earned by renewables. Although high-price periods may still occur (e.g., during evening ramps), renewable generators cannot fully capture these price spikes due to limited output. This gradual erosion of wholesale-market revenues provides the central motivation for supplementing renewable generators’ income with credit-based revenues.

\subsection{Utility No-Banking Condition}

In regulated electricity markets, utilities generally operate under cost-of-service or other profit-neutral regulatory frameworks. Because electricity delivery is a natural-monopoly service, utilities are typically government-owned, nonprofit, or privately owned but subject to strict regulation that limits earnings to approved cost recovery and an authorized return~\citep{joskow_incentive_2008,borenstein_economics_2016}. In this setting, utilities cannot use VRPs to accumulate discretionary profits; the role of the program is instead to recover eligible costs and support renewable expansion at current decision period. This motivates our formulation in which the utility’s decision problem is constrained by per-period financial feasibility, while the objective is to maximize renewable capacity expansion.

\begin{condition}\label{ass:profit_neutral}
\textit{(No intertemporal banking)}\;
For each period \(t\in\mathcal{T}\), all financial and delivery conditions must be satisfied \emph{within} the period. Neither monetary surpluses/deficits nor renewable credits may be carried across periods. In particular:
\begin{enumerate}
    \item \emph{Integrated per-period revenue adequacy:}
    \[
        C(Q_t)+k_t q_t \le R(p_t,Q_t),
        \qquad \forall t\in\mathcal T,
    \]
    so that all eligible non-investment costs and current-period expansion costs must be covered by contemporaneous VRP revenue, no banking or borrowing revenue is allowed.

    \item \emph{Credit-sales feasibility (per-period):}
    \[
        D\!\left(\frac{p_t}{e(Q_t)}\right) \le f(Q_t),
        \qquad \forall t\in\mathcal T.
    \]
    Thus, credits sold in period \(t\) must be backed by renewable electricity delivered in the same period; no banking or borrowing of credits is allowed.
\end{enumerate}
\end{condition}

Condition~\ref{ass:profit_neutral} captures two core features of regulated utility finance. First, the VRP administered by utility is \emph{profit-neutral}: VRP revenues serve cost recovery rather than discretionary profit accumulation, so the relevant objective is renewable expansion rather than profit maximization. Second, feasibility is imposed on a \emph{no-banking} basis: neither financial surpluses/deficits nor renewable credits may be carried across time. This reflects the regulatory structure of utility business models, in which program revenues must remain tied to contemporaneous cost recovery and delivered service. The formulation is therefore institutionally grounded and regulatory-consistent.

\newpage
\section{Single-Period Analysis: Integrated Utility Benchmark}

In this section, we analyze a single-period benchmark in which the electric utility is vertically integrated, meaning that generation, system operation, and program administration are consolidated within a single firm and evaluated on a single financial balance sheet. Vertically integrated utilities recover both generation-related costs and system-level costs through regulated tariffs under cost-of-service or profit-neutral regulation, and therefore internalize the full economic tradeoff between pricing the renewable and expanding capacity. This business model remains prevalent in large parts of the United States—particularly in the Southeast and Midwest—where investor-owned and public utilities continue to operate as vertically integrated monopolies rather than participating in fully restructured wholesale markets (see, e.g., \citet{joskow_incentive_2008, borenstein_economics_2016}).

Modeling the utility as vertically integrated allows us to formulate a single-period optimization problem with a single financial feasibility constraint, which captures total system and generation costs in aggregate. This benchmark serves two purposes. First, it delivers sharp analytical results for optimal pricing and capacity expansion in a static setting. Second, it establishes a clean reference point for later sections, where we relax the integrated balance-sheet assumption to study revenue sharing across separate financial accounts and extend the analysis to multi-period settings in which myopic single-period decisions can be shown to be globally optimal under suitable conditions.

\subsection{Reduction from Multi-Period to Single-Period Formulation}

We consider a single period as a temporally neutral representation of the multi-period model. Because the profit-neutral regulatory condition enforces feasibility independently in each period, the optimization problem in every period \(t\in\mathcal{T}\) can be solved separately, with intertemporal linkage arising only through the capacity state transition \(Q_{t+1}=Q_t+q_t\). Hence, the single-period formulation corresponds to the optimization problem faced by the utility in any generic period, given the existing renewable capacity \(Q\). Dropping time subscripts makes the analysis applicable to all periods.

To streamline notation, define (i) the per-period VRP revenue
\[
R(p,Q)\;:=\;p\,D\!\left(\frac{p}{e(Q)}\right),
\]
and (ii) the integrated non-investment cost term
\[
C(Q)\;:=\;C_S(Q)+C_R(Q)-f(Q)\pi(Q),
\]
so that the integrated financial requirement is \(C(Q)+k q\), where \(k\) is the unit investment cost in the period under consideration.%
\footnote{All primitives and units are summarized in Table~\ref{table:notations}.}

The utility’s objective is to maximize the incremental renewable capacity \(q\), subject to contemporaneous renewable deliverability and an integrated financial feasibility constraint. The resulting single-period problem is:
\begin{subequations}\label{eq:opt_problem_single}
\begin{align}
    \max_{p, q} \quad & q \label{eq:opt_problem_single_obj} \\[4pt]
    \text{s.t.} \quad 
    & D\!\left(\frac{p}{e(Q)}\right) \le f(Q), \label{eq:opt_problem_single_a} \\[4pt]
    & C(Q)+k q \le R(p,Q), \label{eq:opt_problem_single_b} \\[4pt]
    & q \ge 0,\quad p \ge 0. \label{eq:opt_problem_single_c}
\end{align}
\end{subequations}

This single-period problem captures the core decision environment of the program under a vertically integrated benchmark: given the system state \(Q\) and market primitives \(\{D(\cdot),e(\cdot),f(\cdot),\pi(\cdot),C_R(\cdot),C_S(\cdot),k\}\), the utility determines the premium price \(p\) and renewable expansion \(q\) that jointly satisfy feasibility and financial balance within the same period.

\subsection{Phase Structure of the Integrated Single-Period Solution}

Theorem~\ref{thm:single_period_pstar} characterizes the unique optimal price for a given system state \(Q\).

\begin{theorem}[Single-Period Optimal Pricing]
\label{thm:single_period_pstar}
If the optimal solution exhibits strictly positive expansion \(q^*(Q)>0\), then the optimal price \(p^*(Q)\) exists, is unique, and is given by the following two-regime closed form:
\begin{equation}
p^*(Q)=
\begin{cases}
\dfrac{e(Q)}{\epsilon}, 
& \text{if } D\!\left(\dfrac{1}{\epsilon}\right)\le f(Q),\\[10pt]
\dfrac{e(Q)}{\epsilon}\ln\!\left(\dfrac{M}{f(Q)}\right),
& \text{if } D\!\left(\dfrac{1}{\epsilon}\right)> f(Q).
\end{cases}
\label{eq:thm1_pstar}
\end{equation}
\end{theorem}


\emph{Intuition and sketch of proof.}
The single-period optimization problem admits a transparent decision logic. Although the utility’s objective is to maximize incremental renewable capacity \(q\), expansion is constrained entirely by contemporaneous program revenue. As a result, the problem can be decomposed into three steps.

First, for a given renewable capacity \(Q\), the utility chooses the program price \(p\) to maximize revenue \(R(p,Q)\), subject to the deliverability constraint that program sales cannot exceed delivered renewable output. This step is independent of costs and the expansion decision and determines the maximal revenue that can be extracted from the voluntary demand market.

Second, the integrated financial constraint requires that revenue cover the fixed cost term \(C(Q)\). When the financial constraint binds, any residual revenue beyond \(C(Q)\) becomes available for capacity expansion.

Third, residual revenue is converted one-for-one into renewable capacity through the linear investment cost \(k\), yielding a unique expansion level \(q^*\).

\begin{corollary}[Single-period expansion under the optimal price]
\label{cor:qstar}
Given the optimal price \(p^*(Q)\) in Theorem~\ref{thm:single_period_pstar}, the optimal expansion level is determined by the binding integrated financial constraint:
\begin{equation}
q^*(Q)
=
\frac{R\!\big(p^*(Q),Q\big)-C(Q)}{k}.
\label{eq:thm1_qstar}
\end{equation}
\end{corollary}
\emph{Intuition and sketch of proof.} Substituting the closed-form \(p^*(Q)\) from \eqref{eq:thm1_pstar} yields an explicit expression for \(q^*(Q)\) in each regime.

This reduction implies that the equilibrium price is characterized by the revenue-maximizing solution, possibly adjusted by the binding renewable deliverability constraint. The resulting price admits the closed form in \eqref{eq:thm1_pstar}, establishing uniqueness. Given this price, \(q^*\) follows immediately from \eqref{eq:thm1_qstar}. The formal argument follows from analyzing the Karush--Kuhn--Tucker (KKT) conditions, provided in Appendix~A.

\begin{corollary}[Optimal Price Decreases with Renewable Capacity]
Suppose the demand sensitivity parameter \(\epsilon\) is fixed. Since the optimal price satisfies \(p^*(Q)\propto e(Q)\) and the grid-average emissions intensity \(e(Q)\) is strictly decreasing in cumulative renewable capacity \(Q\), it follows that \(p^*(Q)\) decreases monotonically with \(Q\).
\end{corollary}

Although the proof is immediate, the economic implication is significant. When renewable penetration is low (small \(Q\)), the equilibrium premium \(p^*\) is high, meaning that each unit of renewable electricity can command a substantial payment. This highlights the central role of VRP markets in supporting early-stage renewable deployment: high premiums provide strong financial incentives for initial capacity expansion.

As cumulative renewable capacity increases and emissions intensity declines, the optimal premium falls. Consequently, the marginal financial support provided by the VRP demand diminishes with scale. This result formalizes a widely observed real-world pattern: voluntary renewable programs are most effective during early phases of decarbonization, while their influence naturally attenuates as the system approaches low-emissions intensity.

While our baseline assumes exponential demand in the scaled premium $p/e(Q)$, the core single-period result does not depend on this specific functional form. Proposition~\ref{prop:general_pricing_reduction} replaces the closed-form exponential specification with a generic per-period revenue function $R(p,Q)$, requiring only continuity in $p$ and the natural market condition $R(p,Q)\to 0$ as $p\to\infty$. Detailed implications see Section 7.

\begin{proposition}[General Pricing Reduction]\label{prop:general_pricing_reduction}
Let $\mathcal P(Q)\subseteq[0,\infty)$ denote a nonempty closed feasible price set, and let $R:[0,\infty)\times \mathbb{R}_+\to\mathbb{R}_+$ denote per-period VRP revenue. Assume:
\begin{enumerate}
\item[(A1)] (\emph{Continuity}) $p\mapsto R(p,Q)$ is continuous on $\mathcal P(Q)$;
\item[(A2)] (\emph{Vanishing revenue at high prices}) $\lim_{p\to\infty} R(p,Q)=0$;
\end{enumerate}
If the optimal solution of the single-period problem satisfies $q^*(Q)>0$, then optimal pricing maximizes revenue,
\[
p^*(Q)\in \arg\max_{p\in\mathcal P(Q)} R(p,Q),
\]
\end{proposition}

\emph{Intuition and sketch of proof.}
Fix $Q$. Because $R(\cdot,Q)$ is continuous on the nonempty closed set $\mathcal P(Q)\subseteq[0,\infty)$ and $\lim_{p\to\infty}R(p,Q)=0$, the revenue maximization problem $\max_{p\in\mathcal P(Q)}R(p,Q)$ attains at least one maximizer (the vanishing tail implies no maximizing sequence can escape to infinity). In the single-period expansion problem, price affects feasibility only through the available revenue $R(p,Q)$, while the objective is to maximize $q$. When $q^*(Q)>0$, the financial constraint must bind; thus maximizing $q$ is equivalent to maximizing $R(p,Q)$ over feasible prices. Hence any optimal solution selects $p^*(Q)\in\arg\max_{p\in\mathcal P(Q)}R(p,Q)$.

\subsection{Long-Run Capacity Limit and Market-Driven Equilibrium}

Having characterized the optimal single-period pricing and expansion policy, we now turn to the long-run implications of the VRP program. The final regime corresponds to a no-expansion equilibrium, where the optimal expansion decision satisfies \(q^*(Q)=0\). Conceptually, this state represents the \emph{maximum achievable impact} of a voluntary, market-driven renewable expansion mechanism: all financially viable low-carbon capacity has been deployed, and further expansion ceases because program revenues are exactly exhausted by system costs.

Recall that under the integrated formulation, the optimal single-period expansion is given by
\[
q^*(Q)
= \frac{R(Q)-C(Q)}{k}
,
\]
The long-run capacity limit is therefore characterized by the stopping condition
\[
q^*(Q)=0
\quad\Longleftrightarrow\quad
C(Q)=R(Q)=p^*(Q)D(\frac{p^*}{e(Q)})'=\frac{M}{e\,\epsilon}\,e(Q).
\]

\begin{theorem}[Existence and uniqueness of the long-run capacity limit]\label{thm:long_run_capacity}
Assume there exists \(Q^\dagger\) such that the deliverability constraint~\ref{eq:opt_problem_single_a} is non-binding for all \(Q\ge Q^\dagger\). Suppose:
\begin{itemize}
    \item \(C(Q)\) is strictly increasing for all \(Q\ge Q^\dagger\) and \(\lim_{Q\to\infty}C(Q)=+\infty\);
    \item \(C(Q^\dagger)\le \dfrac{M}{e\,\epsilon}\,e(Q^\dagger)\) to ensure financial feasibility at \(C(Q^\dagger)\).
\end{itemize}
Then there exists a unique \(Q^*\ge Q^\dagger\) such that \(q^*(Q^*)=0\).
Moreover, \(Q^*\) is uniquely characterized by the equation
\begin{equation}
C(Q)=\frac{M}{e\,\epsilon}\,e(Q).
\label{eq:Qstar_equation}
\end{equation}
\end{theorem}

\emph{Intuition and sketch of proof.}
For sufficiently large \(Q\) (in particular, for \(Q\ge Q^\dagger\)), the deliverability constraint does not bind, and the optimal price is given by \(p^*(Q)=\frac{e(Q)}{\epsilon}\), yielding
\[
R(Q)=p^*(Q)\,D\!\left(\frac{p^*(Q)}{e(Q)}\right)
=\frac{M}{e\,\epsilon}\,e(Q).
\]
At the no-expansion equilibrium, program revenue must exactly cover total net cost, implying \(R(Q)=C(Q)\).
Because \(C(Q)\) is strictly increasing while \(e(Q)\) (and hence \(R(Q)\)) is strictly decreasing in \(Q\), the two curves intersect at most once. Existence follows from the boundary conditions \(C(Q^\dagger)\le R(Q^\dagger)\) and \(\lim_{Q\to\infty}C(Q)=+\infty\), which guarantee a unique intersection by the intermediate value theorem.

\begin{proposition}[Non-vanishing emissions at the long-run limit]
At the long-run capacity limit \(Q^*\), the grid emissions intensity satisfies \(e(Q^*)>0\).
\end{proposition}

\emph{Proof sketch.}
If \(e(Q^*)=0\), then \(p^*(Q^*)=0\) and hence \(R(Q^*)=0\). Since \(C(Q)\ge 0\) and is strictly positive for nontrivial systems, the equality \(R(Q^*)=C(Q^*)\) cannot hold. Therefore, \(e(Q^*)>0\).

\emph{Interpretation.}
The equilibrium \(Q^*\) represents the \emph{long-run capacity limit} attainable through voluntary renewable program markets alone. Because \(e(Q^*)>0\), the system does not reach full decarbonization at this limit: residual emissions persist even when all market-driven expansion opportunities are exhausted. This establishes a fundamental boundary for voluntary mechanisms and highlights the necessity of complementary regulatory or policy interventions to achieve net-zero outcomes.

\begin{corollary}[Market-driven nature of the long-run capacity limit]
The equilibrium capacity \(Q^*\) depends on market parameters \((M,\epsilon)\) and structural system properties \((C,e,f,\pi)\), but is independent of the renewable deployment cost parameter \(k\).
\end{corollary}

\emph{Interpretation.}
The deployment cost \(k\) affects only the \emph{speed} of expansion—i.e., the magnitude of \(q^*(Q)\) when \(q^*(Q)>0\)—but not the terminal capacity level \(Q^*\). A larger market size \(C\) or lower demand sensitivity \(\epsilon\) increases the attainable equilibrium capacity, whereas changes in \(k\) merely rescale the rate at which the system approaches this limit.

\section{Multi-Period Analysis and Optimality of Myopic Policies}

Having characterized the single-period problem and its KKT structure, we now return to the multi-period setting in which renewable capacity accumulates over time. Let \(Q_t\) denote the installed renewable capacity at period \(t\). Let \(Q_{\mathrm{init}}:=Q_0\) denote the initial renewable capacity. For a utility operating under profit-neutral regulation, new renewable builds are added to an aggregate capacity stock and operated collectively.

\begin{remark}[Markovian structure and profit-neutral regulation]
The system admits a Markovian structure because both the state transition and the per-period feasibility constraints depend only on the current capacity state \(Q_t\). Given the transition rule
\[
    Q_{t+1} = Q_t + q_t
\]
and a per-period stage cost of \(-q_t\), the process \(\{Q_t\}_{t \ge 0}\) forms a controlled Markov process: feasibility and payoffs at time \(t\) depend on past decisions only through \(Q_t\).

Importantly, this Markovian property is reinforced by the \emph{profit-neutral, no-banking regulation} introduced earlier. Because both \emph{revenue adequacy} and \emph{renewable deliverability} must hold within each period, no additional state variables (such as accumulated monetary surplus/deficit or banked credits) carry over across periods. As a result, the system’s intertemporal dynamics are governed entirely by the physical state variable \(Q_t\), with no intertemporal financial coupling.
\end{remark}

\subsection{Myopic Policy Optimality under Monotone Reachability}

We now analyze the multi-period deployment problem directly through the evolution of the capacity state. Let
\[
\bar q(Q)
\;:=\;
\max\Big\{q\ge 0:\ \exists\,p\ge 0 \text{ such that }
D\!\left(\frac{p}{e(Q)}\right)\le f(Q),\ 
C(Q)+kq\le R(p,Q)\Big\}
\]
denote the maximal feasible one-step expansion at state \(Q\). Define the associated reachability map
\begin{equation}\label{eq:monotone_reach}
    S(Q)\;:=\;Q+\bar q(Q).
\end{equation}

The myopic policy is defined by
\[
q_t^{\mathrm{myo}}
\;=\;
\min\{\bar q(Q_t),\,Q^\star-Q_t\},
\]
that is, at each period the utility expands to the maximal feasible level without exceeding the terminal capacity \(Q^\star\).

\begin{theorem}[Myopic policy optimality under monotone reachability]
\label{thm:myopic_clean}
Assume feasible policies satisfy \(0\le q_t \le Q^\star-Q_t\) (\emph{No overbuild}) and the mapping \(S(Q)=Q+\bar q(Q)\) is non-decreasing on \([Q_{\mathrm{init}},Q^\star)\) (\emph{Monotone reachability}).
Then the myopic policy is global optimal, weakly dominates any feasible policy \(\mathcal P\): \begin{equation}
Q_t^{\mathrm{myo}} \;\ge\; Q_t^{\mathcal P}
\end{equation}
\end{theorem}

\emph{Proof sketch.}
Consider any feasible policy $\mathcal P$ and compare the capacity trajectory it generates with that of the myopic policy. The goal is to show that the myopic trajectory weakly dominates the trajectory generated by any feasible policy, in the sense that renewable capacity under the myopic policy is always at least as large as that under $\mathcal P$ until the terminal capacity $Q^\star$ is reached. Formally, we establish by induction that
\begin{equation}
Q_t^{\mathrm{myo}} \;\ge\; Q_t^{\mathcal P}
\qquad
\text{for all } t \text{ prior to hitting } Q^\star.
\label{eq:dominance_clean}
\end{equation}

We begin with the initial period. At $t=0$, both trajectories start from the same initial capacity level, so that
\[
Q_0^{\mathrm{myo}} = Q_0^{\mathcal P} = Q_{\mathrm{init}}.
\]
Thus the dominance relation holds trivially at the starting point.

Next consider an arbitrary period $t$ and suppose that the dominance condition holds at that period, i.e., $Q_t^{\mathrm{myo}} \ge Q_t^{\mathcal P}$, while the myopic trajectory has not yet reached the terminal level $Q^\star$. We now examine how the two policies evolve to period $t+1$. Because policy $\mathcal P$ is feasible, the expansion decision $q_t^{\mathcal P}$ cannot exceed the maximal feasible expansion at the current state $Q_t^{\mathcal P}$. By definition of $\bar q(\cdot)$, this implies
\[
Q_{t+1}^{\mathcal P}
=
Q_t^{\mathcal P}+q_t^{\mathcal P}
\;\le\;
Q_t^{\mathcal P}+\bar q(Q_t^{\mathcal P})
=
S(Q_t^{\mathcal P}),
\]
where $S(Q)=Q+\bar q(Q)$ denotes the maximal reachable capacity in one step from state $Q$.

The myopic policy, by construction, expands to the largest feasible level in each period subject to the upper bound $Q^\star$. Therefore its state transition satisfies
\[
Q_{t+1}^{\mathrm{myo}}
=
\min\{S(Q_t^{\mathrm{myo}}),\,Q^\star\}.
\]

The key step follows from the induction hypothesis together with the monotonicity of the reachability function $S(\cdot)$. Since $Q_t^{\mathrm{myo}} \ge Q_t^{\mathcal P}$ and $S(\cdot)$ is nondecreasing, we obtain
\[
S(Q_t^{\mathrm{myo}})
\;\ge\;
S(Q_t^{\mathcal P}).
\]
Combining this inequality with the previous bounds yields
\[
Q_{t+1}^{\mathrm{myo}}
\;\ge\;
\min\{Q_{t+1}^{\mathcal P},\,Q^\star\}.
\]
Consequently, as long as the trajectory under policy $\mathcal P$ has not yet reached the terminal capacity (i.e., $Q_{t+1}^{\mathcal P} < Q^\star$), it follows that
\[
Q_{t+1}^{\mathrm{myo}} \ge Q_{t+1}^{\mathcal P}.
\]
This establishes the induction step and shows that the myopic trajectory weakly dominates any feasible trajectory until the terminal capacity $Q^\star$ is reached.

\subsection{Implications of Myopic Optimality}

We now state two immediate consequences of Theorem~\ref{thm:myopic_clean}. Both follow directly from the statewise dominance property established in the proof.

\begin{corollary}[Minimum hitting time]
\label{cor:hitting_time}
Under the assumptions of Theorem~\ref{thm:myopic_clean}, the myopic policy reaches the terminal capacity level \(Q^\star\) in the fewest periods among all feasible policies.
\end{corollary}

\emph{Proof sketch.}
By statewise dominance,
\(
Q_t^{\mathrm{myo}} \ge Q_t^{\mathcal P}
\)
for all \(t\) prior to hitting \(Q^\star\).
Therefore, if a feasible policy \(\mathcal P\) reaches \(Q^\star\) at time \(T\), the myopic trajectory must satisfy
\(Q_T^{\mathrm{myo}} \ge Q_T^{\mathcal P} \ge Q^\star\),
implying that the myopic policy reaches \(Q^\star\) no later than \(T\).
\hfill\(\square\)

\begin{corollary}[Minimum cumulative emissions]
\label{cor:cumulative_emissions}
If the emissions intensity function \(e(Q)\) is nonincreasing in \(Q\), then among all feasible policies the myopic policy minimizes cumulative emissions up to the hitting time of \(Q^\star\).
\end{corollary}

\emph{Proof sketch.}
From statewise dominance,
\(Q_t^{\mathrm{myo}} \ge Q_t^{\mathcal P}\) prior to hitting \(Q^\star\).
Since \(e(Q)\) is nonincreasing, it follows that
\(e(Q_t^{\mathrm{myo}}) \le e(Q_t^{\mathcal P})\) period by period.
Summing over time yields weakly lower cumulative emissions under the myopic policy.
\hfill\(\square\)

\paragraph{Policy significance.}
Theorem~\ref{thm:myopic_clean} and its corollaries deliver a simple operational rule for a utility operating under profit-neutral regulation: \emph{build to the maximum feasible level each period}. Pricing decisions need only ensure feasibility of this maximal build; no intertemporal smoothing or long-horizon forecasting is required. In practice, this shifts managerial focus from solving dynamic optimization problems to identifying and relaxing feasibility bottlenecks, such as interconnection limits, siting constraints, procurement processes, and regulatory approvals, that restrict \(\bar q(Q)\).

\section{Revenue Sharing under Separated Financial Constraints}\label{sec:revenue_sharing}

We now consider a decision problem under an \emph{unbundled utility structure}, also known as as \emph{Independent Power Producers (IPPs)} business model, in which renewable generation and utility system operation are financially separated. In this setting, renewable generators operate as independent entities, while the utility administers the VRP, sets the premium price, and invests in system-level integration and expansion. Unlike the vertically integrated benchmark, the utility no longer internalizes generation costs on a single balance sheet; instead, financial feasibility must be satisfied separately for the operator and the generators.

\subsection{Single-Period Optimization with Revenue Sharing}

To capture this institutional separation, we introduce an explicit revenue-sharing decision variable \(\gamma \in [0,1]\), which determines how total VRP revenue is allocated between the two parties. Specifically, a fraction \(\gamma\) of total program revenue is transferred to renewable generators, while the remaining fraction \(1-\gamma\) is retained by the operator. Importantly, \(\gamma\) is not a market outcome but a \emph{policy-controlled parameter}, reflecting contract design or regulatory rules governing revenue allocation within the VRP. Recall the two cost aggregates
\[
C_1(Q,q):=C_S(Q)+k\,q,
\qquad
C_2(Q):=C_R(Q)-f(Q)\pi(Q),
\]
where all primitives are summarized in Table~\ref{table:notations}.
The single-period optimization problem under revenue sharing is:
\begin{subequations}\label{eq:opt_problem_single_rs}
\begin{align}
    \max_{p, q, \gamma} \quad & q \label{eq:opt_problem_single_rs_obj} \\[4pt]
    \text{s.t.} \quad 
    & D\!\left(\frac{p}{e(Q)}\right) \le f(Q), \label{eq:opt_problem_single_rs_a} \\[4pt]
    & C_1(Q,q) \le (1-\gamma)\,R(p,Q), \label{eq:opt_problem_single_rs_b} \\[4pt]
    & C_2(Q) \le \gamma\,R(p,Q), \label{eq:opt_problem_single_rs_c} \\[4pt]
    & q \ge 0,\quad p \ge 0,\quad \gamma \in [0,1]. \label{eq:opt_problem_single_rs_d}
\end{align}
\end{subequations}

Constraint~\eqref{eq:opt_problem_single_rs_a} enforces renewable deliverability: renewable credits sold by the program must be backed by delivered renewable electricity in the same period. Constraint~\eqref{eq:opt_problem_single_rs_b} ensures that the operator’s retained share of program revenue covers system costs and any new investment undertaken in the period. Constraint~\eqref{eq:opt_problem_single_rs_c} guarantees generator viability by requiring that the transferred program revenue cover the generator-side net cost \(C_2(Q)\).

Relative to the integrated benchmark, the key difference is the explicit separation of financial constraints through \(\gamma\). This separation introduces additional flexibility in revenue allocation but does not, as we show next, alter the optimal pricing signal or the long-run capacity limit. Instead, \(\gamma\) governs how program revenue is distributed between generator support and system expansion, making it the primary policy lever in unbundled utility environments.

\subsection{Optimal Revenue Sharing Policy}

Introducing revenue sharing changes the institutional structure of the problem, but it need not distort the underlying market-based pricing signal or the expansion outcome. In this subsection, we characterize the operator’s optimal revenue-sharing rule and show that, whenever the revenue-sharing choice is interior, the unbundled problem collapses to the integrated benchmark.

\begin{proposition}[Optimal revenue-sharing rule]
\label{prop:gamma_star}
Fix a state \(Q\) and suppose the optimal solution exhibits strictly positive expansion \(q^*(Q)>0\). Let \(p^*(Q)\) denote the optimal price in the integrated formulation, and let \(R^*(Q):=R(p^*(Q),Q)\) denote the corresponding program revenue. Then an optimal revenue share is given by
\begin{equation}
\gamma^*(Q)=\max\left\{0,\frac{C_2(Q)}{R^*(Q)}\right\}.
\label{eq:gamma_star}
\end{equation}
In particular, if \(C_2(Q)\le 0\) (generator surplus from the energy market), then \(\gamma^*(Q)=0\). If \(C_2(Q)>0\) and \(\gamma^*(Q)\in(0,1)\), then the generator viability constraint binds at the optimum:
\[
C_2(Q)=\gamma^*(Q)\,R^*(Q).
\]
\end{proposition}

\emph{Sketch of proof.}
The result follows from the KKT conditions of \eqref{eq:opt_problem_single_rs}. If \(C_2(Q)\le 0\), the generator viability constraint is slack at \(\gamma=0\), so allocating any positive share to generators weakly reduces the operator’s retained revenue and cannot increase expansion, implying \(\gamma^*(Q)=0\). Otherwise, when \(C_2(Q)>0\) and \(\gamma\) is interior, complementary slackness implies the generator viability constraint binds, yielding \(\gamma=C_2(Q)/R(p,Q)\). Evaluating at the expansion-maximizing price \(p^*(Q)\) gives \eqref{eq:gamma_star}. \hfill\(\square\)

The rule above identifies an expansion-maximizing revenue share for the single-period separated-accounts problem: conditional on feasibility, any alternative \(\gamma\) cannot increase the attainable expansion \(q\). We next show that, in the interior regime \(\gamma^*(Q)\in(0,1)\), the two binding financial constraints aggregate to the integrated feasibility constraint, implying the same optimal pricing and expansion outcomes as in the integrated benchmark.

\begin{remark}
The boundary case \(\gamma^*(Q)=1\) is excluded by feasibility. Since the utility bears strictly positive institutional cost in administering the VRP program, allocating all revenue to generators would leave zero retained revenue for the operator in \eqref{eq:opt_problem_single_rs_b}, making the operator-side constraint infeasible. Therefore, any feasible solution must satisfy \(\gamma^*(Q)<1\).
\end{remark}

\begin{proposition}[Equivalence to the integrated benchmark under interior revenue sharing]
\label{prop:interior_equivalence}
Fix \(Q\) and suppose the optimal revenue-sharing choice is interior, i.e., \(\gamma^*(Q)\in(0,1)\). Then:
\begin{enumerate}
    \item The optimal price and expansion under \eqref{eq:opt_problem_single_rs} coincide with the integrated benchmark solution, i.e., \(p^{\mathrm{RS}}(Q)=p^*(Q)\) and \(q^{\mathrm{RS}}(Q)=q^*(Q)\).
    \item Consequently, all analytical results derived under the integrated benchmark remain valid, including the single-period closed-form solution (Theorem~\ref{thm:single_period_pstar}), the long-run capacity limit characterization (Theorem~\ref{thm:long_run_capacity}), the multi-period formulation, and the global optimality of myopic policies (Theorem~\ref{thm:myopic_clean}), provided the interior condition holds along the realized capacity path.
\end{enumerate}
\end{proposition}

\emph{Sketch of proof.}
When \(\gamma^*(Q)\in(0,1)\) and \(q^*(Q)>0\), complementary slackness implies that both separated financial constraints bind at the optimum:
\[
C_1(Q,q)=(1-\gamma)\,R(p,Q),
\qquad
C_2(Q)=\gamma\,R(p,Q).
\]
Adding these equalities yields
\[
C_1(Q,q)+C_2(Q)=R(p,Q),
\]
which is exactly the integrated financial feasibility constraint. Since the renewable deliverability constraint is identical in both formulations, the unbundled single-period problem reduces to the integrated benchmark. Therefore, the same revenue-maximizing pricing rule applies, and the implied expansion (given by the binding integrated constraint) coincides as well. The multi-period and myopic optimality results follow because the per-period feasible set and induced expansion mapping \(\bar q(Q)\) are unchanged whenever the reduction holds. \hfill\(\square\)

\subsection{Phase Transition Dynamics}

We now examine how the optimal revenue-sharing rule translates into distinct operational regimes as renewable penetration increases. 
Because pricing is already determined by the integrated benchmark, the only remaining policy lever is the allocation of total program revenue between generator support and system expansion.

Let \(R^*(Q):=R(p^*(Q),Q)\) denote total program revenue at the optimal price.

\begin{corollary}[Expansion under revenue sharing]
\label{lem:q_gamma_clean}
Fix a state \(Q\) and suppose the optimal price \(p^*(Q)\) is applied.
If expansion is positive and both financial constraints bind, the operator-funded expansion level is given by
\begin{equation}
q(\gamma;Q)
=
\frac{(1-\gamma)R^*(Q)-C_S(Q)}{k}.
\label{eq:q_gamma}
\end{equation}
In particular, expansion is strictly decreasing in \(\gamma\).
\end{corollary}

\emph{Proof sketch.}
When expansion is positive, the operator’s budget constraint binds:
\[
C_S(Q)+kq=(1-\gamma)R^*(Q).
\]
Solving for \(q\) yields \eqref{eq:q_gamma}. Monotonicity in \(\gamma\) follows immediately.
\hfill\(\square\)

Corollary~\ref{lem:q_gamma_clean} formalizes a simple economic principle: program revenues must first cover any renewable revenue shortfall; only the residual can finance new capacity additions. Since the objective is to maximize \(q\), the optimal revenue-sharing rule allocates the minimum share required to ensure generator viability, directing all remaining revenue toward expansion.

\paragraph{Phase structure.}
As renewable capacity increases, the binding patterns of the financial constraints generate three qualitatively distinct regimes:

\begin{itemize}
    \item \textbf{Phase 1: Spontaneous expansion (\(\gamma=0,\; q>0\)).}  
    When renewable generators are viable from wholesale-market revenue alone (\(C_2(Q)\le 0\)), no revenue transfer is required. All program revenue is retained by the operator and directed toward expansion. This phase typically occurs at low penetration levels, when energy revenue and emission differentials remain high.

    \item \textbf{Phase 2: Revenue-supported expansion (\(\gamma>0,\; q>0\)).}  
    As renewable penetration increases, declining wholesale prices and rising curtailment reduce generator net revenue. A positive revenue share becomes necessary to maintain viability. Expansion continues, but a growing portion of program revenue is absorbed by operating support rather than new builds.

    \item \textbf{Phase 3: Long-run equilibrium (\(\gamma>0,\; q=0\)).}  
    Eventually, total program revenue is just sufficient to cover operating costs. No residual remains to finance additional capacity, and expansion ceases. The system reaches the long-run capacity limit characterized earlier.
\end{itemize}

These regimes arise endogenously from the interaction between declining emission intensity, changing wholesale-market revenues, and the revenue-neutral constraint. 

\paragraph{Policy interpretation.}
The phase structure highlights an important timing implication for VRPs. In early stages, when \(\gamma=0\), the program generates its largest expansion effect because all program revenue finances new capacity. As penetration rises, increasing revenue must be diverted to sustain existing generation, reducing marginal expansion. This dynamic underscores the importance of early program deployment: voluntary mechanisms are most effective when renewable penetration remains low and emission differentiation is large.

\section{Numerical Illustration}

\subsection{ISO-NE Test System}

The numerical illustration adopts the 8-bus ISO-NE (Independent System Operator–New England) test system, a standard configuration widely used in power-system economics to benchmark dispatch and market mechanisms. The system consists of eight nodes, twelve transmission lines, and seventy-six thermal generators, with a total installed capacity of approximately 23~GW and an average system load of 13~GW. Transmission congestion in the ISO-NE system is generally rare and limited, reflecting the region’s strong interconnection and centralized coordination under ISO-level management. Because bulk transmission expansion is planned and operated by the ISO rather than by local utilities, our framework focuses on the locational marginal price of a \emph{single-bus representation} from the 8-bus system to capture the economic behavior of an individual utility within its local service territory. This abstraction isolates the retail-facing renewable program from inter-zonal transmission dynamics while preserving realistic nodal prices and system conditions.

Within this localized utility framework, renewable operating costs \(C_R(Q)\) are represented by the operation and maintenance (O\&M) expenses of renewable generators. The unit investment cost of incremental renewable additions is captured by \(k_t\) (or \(k\) in a stationary parameterization), while the operator’s system-side cost \(C_S(Q)\) primarily reflects expenditures associated with battery energy storage deployment and operation, which serve as the main flexibility resources supporting renewable integration and reliability. To capture the increasing marginal difficulty of expansion, we impose moderate convexity on both \(C_R(Q)\) and \(C_S(Q)\), reflecting factors such as higher siting and interconnection costs for remote resources, the need for enhanced control infrastructure, and growing balancing costs as renewable and storage capacities expand. These assumptions ensure that the simulated cost structure realistically mirrors the economic challenges faced by utilities scaling renewables under local operational constraints. The detailed parameterization and quantitative assumptions are presented in Appendix~B.


We now quantify the monotone reachability condition in equation~\ref{eq:monotone_reach}:
\[
S(Q)\;=\;Q+\frac{\frac{M}{e\,\epsilon}e(Q)-C(Q)}{\lambda}.
\]
If $S$ is differentiable, the monotone reachability requirement $S'(Q)\ge 0$ is equivalent to
\begin{equation}
\label{eq:mono_reachability_primitives}
1+\frac{1}{\lambda}\Big(\frac{M}{e\,\epsilon}e'(Q)-C'(Q)\Big)\;\ge\;0,
\qquad \forall Q\in[Q_0,Q^*).
\end{equation}

It shall be acknowledged that emission reduced with renewable capacity \(e'(Q)<0\) and generally non-investment cost increases with renewable capacity \(C'(Q)>0\), so the monotone reachability condition is not holding automatically without quantification check. In our ISO-NE simulation, we use the parameterization \(M=10\), \(\epsilon=0.0045\), and \(k=1000\).%
\footnote{All monetary units are consistent with the model formulation; see Appendix~B for full scaling and unit conversions.}
We also compute uniform upper bounds on the relevant derivatives over the simulated state range (e.g., \(\max|e'(Q)|<0.1\) and \(\max|C'(Q)|<150\)). Under this calibration, the sufficient condition for monotone reachability is satisfied with slack:
\[
1+\frac{1}{\lambda}\Big(\frac{M}{e\,\epsilon}e'(Q)-C'(Q)\Big)\ \ge1+\frac{1}{k}\left(-\frac{M}{e\,\epsilon}\,\max|e'(Q)|-\max|C'(Q)|\right)=0.768>0.
\]
Therefore, the monotone reachability condition holds, implying that the myopic policy is globally optimal for the ISO-NE calibration. Although myopic optimality is not automatic, it is empirically plausible under realistic power-system parameterizations.

\subsection{Numerical Calibration and Result Interpretation}

We calibrate VRP demand using an optimistic but empirically grounded assumption consistent with recent projections for voluntary market growth. Assume approximately one-third of total utility demand—combining corporate and residential customers—is expected to be willing to purchase renewable contracts by 2030. Setting the price premium at around 20~\$/MWh (equivalently, 2~¢/kWh) for this voluntary segment closely aligns with observed surcharges in today’s VRP programs across several U.S. utilities. Based on this willingness-to-pay benchmark, we reverse-engineer the price-sensitivity parameter to match the implied responsiveness under our exponential specification
\[
D\!\left(\frac{p_t}{e(Q_t)}\right)=M\exp\!\left(-\epsilon\frac{p_t}{e(Q_t)}\right).
\]
For the ISO-NE system, assuming a local utility market size of 10~GW capacity out of a total regional generation capacity of 23~GW, this calibration yields \(\epsilon = 0.0045\). 

The parameterization of \(\epsilon\) and its impact on renewable contract pricing—expressed both on a per-capacity (GW) and per-energy (\$/MWh) basis—are summarized in Table~\ref{tab:price_sensitivity}. As noted by \citet{oshaughnessy_corporate_2021}, offtakers may compensate renewable generators through a mix of volumetric and capacity-based rates, where the capacity-based pricing provides a steadier and more predictable revenue stream. This dual representation therefore captures the practical structure of renewable contracts, in which energy-based and capacity-based premiums jointly determine the overall financial incentive for voluntary participation. The resulting \(\epsilon\) reflects a relatively price-tolerant voluntary segment, consistent with the increasing corporate and community procurement of renewable energy observed in recent years.

\begin{table}[h!]
\centering
\caption{Price thresholds for VRP demand at different adoption levels (\(Q=5\)~GW, \(\epsilon=0.0045\), Wind capacity factor \(=0.35\)).}
\begin{tabular}{lcc}
\toprule
\textbf{Demand Level} & \textbf{Price \(p\) [M\$/GW]} & \textbf{Equivalent Price [\$/MWh]} \\
\midrule
$D = 0.33\,M$  & 69.37  & 22.63  \\
$D = 0.10\,M$  & 145.40  & 47.42 \\
$D = 0.05\,M$  & 247.03  & 80.57 \\
\bottomrule
\end{tabular}
\label{tab:price_sensitivity}
\end{table}

\IfFileExists{fig_main/optimal_pricing_demand.png}
  {\typeout{[FIG FOUND] 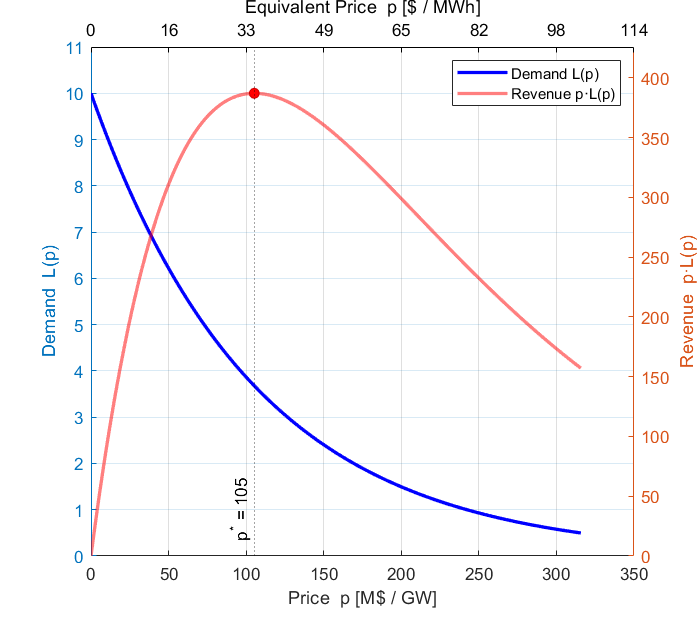}}
  {\typeout{[FIG MISSING] fig_main/optimal_pricing_demand.png}}

\begin{figure}[h!]
    \centering
    \includegraphics[width=0.95\linewidth]{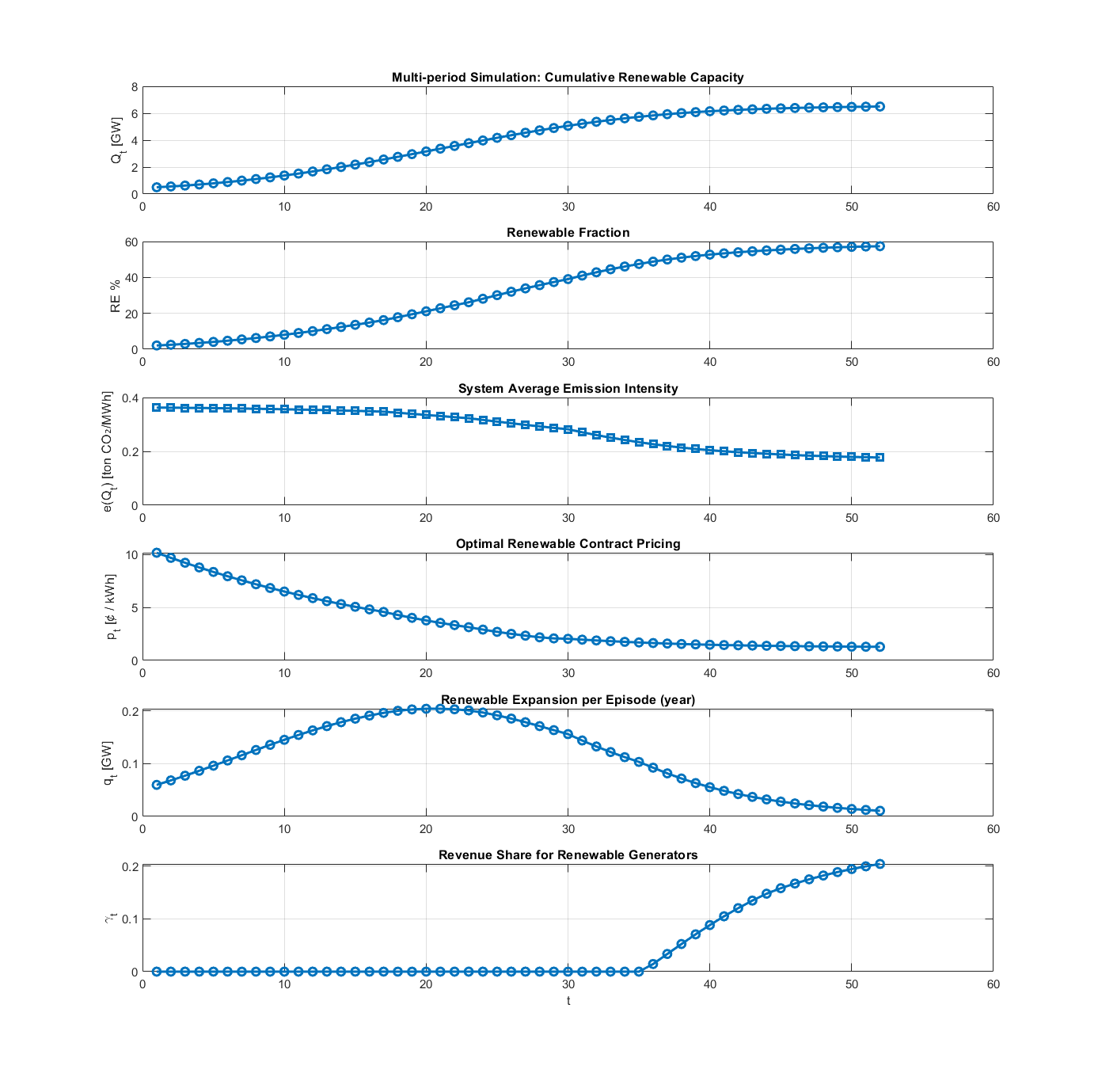}
    \caption{%
    \textbf{Multi-period simulation under baseline demand.}
    The figure summarizes system evolution across multiple decision periods:
    (1) cumulative renewable capacity \(Q_t\),
    (2) modeled renewable (wind) generation percentage,
    (3) grid emissions intensity \(e(Q_t)\),
    (4) optimal renewable program price \(p_t^*\),
    (5) renewable capacity expansion \(q_t^*\),
    and (6) revenue-sharing ratio \(\gamma_t^*\) between the program operator and renewable generators.
    Together, these panels illustrate how VRP revenues can sustain continued renewable deployment while ensuring financial feasibility and declining grid emissions intensity.
    }
    \label{fig:baseline_simulation}
\end{figure}

Figure~\ref{fig:baseline_simulation} summarizes the dynamic behavior of the VRP program. The cumulative renewable capacity \(Q_t\) exhibits the characteristic S-shaped trajectory (Subfigures~1 and~5), consistent with diffusion models such as \citet{bass_new_1969}, where technology adoption accelerates in the early stages and gradually levels off as the market matures. The incremental expansion \(q_t\) first accelerates because, at low renewable penetration levels, generators remain financially self-sufficient from the wholesale energy market, allowing VRP revenues to be fully allocated toward new capacity additions. In this early phase, the rate-limiting factor is the limited volume of renewable available for sale rather than cost recovery, leading to a rapidly increasing expansion rate. As renewable penetration rises, the grid emissions intensity \(e(Q_t)\) declines, reducing the incentive for voluntary differentiation. At the same time, both the operator’s system cost and renewable operating costs increase, gradually slowing expansion. Eventually, investment halts once revenues can no longer cover incremental costs, yielding the observed S-shaped cumulative trajectory.

The grid emissions trajectory (Subfigure~3) decreases monotonically but nonlinearly with \(Q_t\). When renewable capacity interacts with the ISO-NE wholesale market, the fossil generation displaced by renewables follows the merit order of marginal costs rather than emissions intensity, producing a nonlinear decline in emissions intensity. As predicted by the theoretical stopping condition, the terminal emissions intensity \(e(Q^*)\) remains strictly positive: once the grid approaches full decarbonization, the willingness to pay for additional VRPs disappears, and further expansion ceases. Thus, while the programs deepen renewable penetration, they cannot independently achieve a fully net-zero grid.

The renewable premium price \(p_t^*\) (Subfigure~4) declines monotonically over time, consistent with the analytical result that the optimal price decreases as both willingness to pay and emissions differentiation diminish. The revenue-sharing parameter \(\gamma_t^*\) (Subfigure~6) reveals the transition between spontaneous and policy-supported expansion phases. When \(Q_t\) is small, \(\gamma_t^* = 0\), indicating that renewables can recover all costs from the energy market (i.e., \(C_2(Q_t)\le 0\)). The point at which \(\gamma_t^*\) becomes positive marks the onset of the VRP’s active phase: beyond this threshold, VRP revenues are partially redirected to sustain generator viability and to finance the operator’s system-level investments, particularly in battery energy storage.

More specifically, low-carbon generation currently accounts for roughly 39\% of ISO-NE’s total demand, broadly matching our baseline numerical illustration of 5 GW of renewable capacity in a utility system with 10 GW peak demand. This corresponds to approximately period 30 on the x-axis across all subfigures in Figure~\ref{fig:baseline_simulation}. At this point, our model suggests that the system is roughly five years away from the stage at which revenue sharing with renewable generators becomes necessary to sustain their financial balance, due to declining revenues from the energy market. Consistent with current conditions, the model therefore indicates that renewable generators may begin to require support from voluntary renewable procurement programs around 2030.

Quantitatively, the program extends renewable capacity from roughly 5.5~GW to 6.5~GW—a nontrivial increase—but the additional 1~GW requires nearly the same duration (about 30 ~years) as the spontaneous expansion from 0.5~GW to 5.5~GW. This finding underscores that while voluntary mechanisms can meaningfully extend renewable capacity, they do so at a substantially slower marginal rate compared with early-stage growth.

\section{Discussion and Conclusions}

This paper develops an institutionally grounded benchmark for voluntary renewable program (VRP) design under utility-led implementation. The framework is deliberately general: it does not depend on a particular market design, ownership structure, or numerical calibration, but instead isolates a small set of economically meaningful primitives---the emissions-differentiation channel \(e(Q)\), the demand response to the implied abatement price, the cost of renewable expansion, and the regulatory requirement of contemporaneous financial feasibility. These ingredients allow the model to apply broadly across utility-operated voluntary programs, including vertically integrated and unbundled settings coupled with independent power producers, while preserving analytical tractability.

The first central insight is that voluntary renewable programs can meaningfully support renewable expansion, but their long-run effect is inherently limited. VRP demand creates value by monetizing emissions differentiation, thereby generating revenue that can be directed to new renewable additions. However, as renewable penetration rises and grid emissions intensity falls, the value of that differentiation declines endogenously. As a result, voluntary program revenue eventually becomes just sufficient to cover ongoing system and participation requirements, leaving no residual for further expansion. In this sense, voluntary programs can accelerate the transition, but they cannot by themselves complete it. Full decarbonization requires complementary mandatory instruments, such as compliance obligations, clean electricity standards, or carbon pricing, that preserve incentives to abate even when average grid emissions are already low.

The second insight is a separation result: conditional on the renewable-capacity state \(Q\), optimal pricing is pinned down by the revenue problem alone, while expansion is determined by the remaining feasibility margin after costs are covered. Our results are not limited to the exponential demand specification (Proposition~\ref{prop:general_pricing_reduction}). The core single-period analysis extends to a general revenue function $R(p,Q)$ under mild regularity conditions, so the main conclusions continue to hold for alternative demand forms (e.g., logit or isoelastic) and for alternative monotone ways in which $Q$ enters demand through emissions intensity beyond the ratio $p/e(Q)$. For instance, demand may take the form $D\left(g(p,e(Q))\right)$ for a monotone index $g$, or $D(p,e(Q))=\phi(e(Q))\bar D(p)$. In all such cases, whenever the financial constraint binds, the optimal pricing rule remains $p^*(Q)\in\arg\max_{p\in\mathcal P(Q)}R(p,Q)$; and any downstream result that depends only on the revenue envelope $R^\star(Q):=\max_{p\in\mathcal P(Q)}R(p,Q)$ can be restated in terms of that envelope. This result has both analytical and empirical value. Analytically, it clarifies why the pricing problem admits a clean characterization. Empirically, it implies that understanding VRP demand---its level, elasticity, and heterogeneity---is first-order for evaluating program performance, even before detailed engineering-cost calibration is introduced. The demand side determines the attainable revenue frontier, and thus the room available for renewable expansion.

The third insight is dynamic. Under the monotone reachability condition, a simple myopic policy that expands to the per-period feasible limit is globally optimal. This gives the model a strong operational interpretation: when the condition holds, the utility need not solve a complicated long-horizon planning problem or rely on highly uncertain forecasts of future technology costs, demand growth, or market conditions. Instead, it can follow a transparent period-by-period expansion rule and still attain the long-run optimum. This is especially relevant in real program environments, where long-run projections are uncertain and policy credibility often depends on simple, auditable decision rules.

The fourth insight concerns institutional design. In the unbundled setting, optimal revenue sharing is state-dependent. At low renewable penetration, when generators remain viable from wholesale-market revenues, the operator optimally retains the full program revenue and directs it toward additional expansion. As penetration rises and generator-side viability becomes tighter, a growing share of VRP revenue must be allocated to sustain participation, reducing the residual available for new capacity. This phase structure implies that voluntary renewable programs are most powerful when launched early: the same demand base produces greater expansion when emissions differentiation is high and fewer revenues must be diverted toward maintaining incumbent viability. Delayed adoption, by contrast, shifts the role of the program from expansion finance toward partial support of an already transformed system. Consistent with our numerical interpretation, the model suggests that for many markets already approaching relatively high renewable penetration, the revenue-sharing phase may arrive soon as renewable energy-market revenues continue to shrink, and the framework therefore provides analytical support for designing such revenue-sharing policies.

More broadly, the framework highlights a governance requirement for credible voluntary programs. Because the model is profit-neutral and feasibility-based, program effectiveness depends on revenues being transparently allocated toward the modeled objectives rather than diverted elsewhere. Transparent accounting, verification of credit-backed delivery, and credible commitment over revenue use are therefore not secondary implementation details; they are central to preserving demand and sustaining the pricing mechanism itself. A voluntary market can generate meaningful expansion only if participants believe that the premium they pay is translated into real and verifiable system impact.

Taken together, these results position the framework as a benchmark for analyzing utility-operated VRP programs under realistic regulatory constraints. Its main contribution is not to provide a fully exhaustive description of every institutional detail, but to identify the core economic logic linking emissions differentiation, voluntary demand, utility feasibility, and renewable expansion. That logic yields a clear policy conclusion: voluntary renewable programs can improve price transparency, strengthen impact transparency, and accelerate renewable deployment, but their role is fundamentally complementary rather than sufficient. Their greatest value lies in providing an early, credible, and financially disciplined channel for renewable expansion within a broader decarbonization architecture that must ultimately include mandatory policy support.


\bibliographystyle{plainnat}
\bibliography{Renewable_economics_bib_20251103}

\newpage
\section{Appendix A: Math Proofs}

\subsection{Theorem 1: Single-Period Optimal Pricing}

\begin{proof}
Fix a state $Q$ and abbreviate $e:=e(Q)>0$, $f:=f(Q)$, and $C:=C(Q)$. Under exponential demand condition,
\[
D\!\left(\frac{p}{e}\right)=M\exp\!\left(-\epsilon\frac{p}{e}\right),\qquad
R(p,Q)=p\,D\!\left(\frac{p}{e}\right)=pM\exp\!\left(-\epsilon\frac{p}{e}\right).
\]
Rewrite constraints in $\le 0$ form:
\[
h_1(p,q):=D\!\left(\frac{p}{e}\right)-f\le 0,\quad
h_2(p,q):=C+kq-R(p,Q)\le 0,\quad
h_3(p,q):=-q\le 0,\quad
h_4(p,q):=-p\le 0.
\]
Let multipliers $(\lambda,\mu,\nu,\eta)\ge 0$ correspond to $(h_1,h_2,h_3,h_4)$ and define the Lagrangian
\[
\mathcal{L}(p,q;\lambda,\mu,\nu,\eta)
= q-\lambda h_1(p,q)-\mu h_2(p,q)-\nu h_3(p,q)-\eta h_4(p,q).
\]
The KKT conditions are:

\noindent\emph{(i) Primal feasibility:}
\[
D\!\left(\frac{p^*}{e}\right)\le f,\quad C+kq^*\le R(p^*,Q),\quad q^*\ge 0,\quad p^*\ge 0.
\]
\emph{(ii) Dual feasibility:} $\lambda^*,\mu^*,\nu^*,\eta^*\ge 0$.

\noindent\emph{(iii) Complementary slackness:}
\[
\lambda^*\!\left(D\!\left(\frac{p^*}{e}\right)-f\right)=0,\quad
\mu^*\!\left(C+kq^*-R(p^*,Q)\right)=0,\quad
\nu^*q^*=0,\quad
\eta^*p^*=0.
\]
\emph{(iv) Stationarity:} using $D'(x)=-\epsilon D(x)$ and
\[
R_p(p,Q)=\frac{\partial R(p,Q)}{\partial p}
=D\!\left(\frac{p}{e}\right)\!\left(1-\epsilon\frac{p}{e}\right),
\]
we have
\begin{align*}
\frac{\partial \mathcal{L}}{\partial q}
&=1-\mu k+\nu=0,\\
\frac{\partial \mathcal{L}}{\partial p}
&=\lambda\frac{\epsilon}{e}D\!\left(\frac{p}{e}\right)+\mu\,D\!\left(\frac{p}{e}\right)\!\left(1-\epsilon\frac{p}{e}\right)+\eta=0.
\end{align*}

Assume $q^*(Q)>0$. Then $\nu^*=0$ by complementary slackness, hence stationarity in $q$ implies
\[
1-\mu^*k=0\quad\Rightarrow\quad \mu^*=\frac{1}{k}>0.
\]
Moreover, if $C+kq^*<R(p^*,Q)$ held strictly, one could increase $q$ slightly without violating any constraint, contradicting optimality of $q^*$. Hence the financial constraint always binds at optimal solutions:
\[
C+kq^*=R(p^*,Q).
\]
Since $q^*>0$ requires $R(p^*,Q)>C$ and $R(0,Q)=0$, we must have $p^*>0$, so $\eta^*=0$. Substituting $\mu^*=1/k$ and $\eta^*=0$ into stationarity in $p$ and dividing by $D(p^*/e)>0$ yields
\[
\lambda^*\frac{\epsilon}{e}+\frac{1}{k}\left(1-\epsilon\frac{p^*}{e}\right)=0
\quad\Rightarrow\quad
p^*=\frac{e}{\epsilon}+k\lambda^*.
\]
Now there are two cases.

\smallskip
\noindent\textbf{Case 1 (deliverability slack).} If $D(p^*/e)<f$, then complementary slackness implies $\lambda^*=0$, hence
\[
p^*=\frac{e}{\epsilon}.
\]
Feasibility requires $D(1/\epsilon)=D\!\left((e/\epsilon)/e\right)\le f$, which is exactly the first regime condition.

\smallskip
\noindent\textbf{Case 2 (deliverability binding).} If $D(p^*/e)=f$, then solving $M\exp(-\epsilon p^*/e)=f$ gives
\[
p^*=\frac{e}{\epsilon}\ln\!\left(\frac{M}{f}\right),
\]
which applies precisely when the unconstrained maximizer violates deliverability, i.e.\ $D(1/\epsilon)>f$.

It remains to verify optimality and uniqueness. Under $q^*>0$, maximizing $q$ subject to $C+kq\le R(p,Q)$ is equivalent to maximizing $R(p,Q)$ over feasible $p$ (since $q=(R(p,Q)-C)/k$ whenever $R(p,Q)>C$). The revenue derivative is
\[
R_p(p,Q)=D\!\left(\frac{p}{e}\right)\!\left(1-\epsilon\frac{p}{e}\right),
\]
so the unique stationary point of $R(\cdot,Q)$ on $(0,\infty)$ is $p_u=e/\epsilon$, where $R_p(p_u,Q)=0$.
Moreover,
\[
R_{pp}(p,Q)=\frac{\partial^2 R(p,Q)}{\partial p^2}
=-\frac{\epsilon}{e}D\!\left(\frac{p}{e}\right)\left(2-\epsilon\frac{p}{e}\right),
\]
so $R$ is strictly increasing on $(0,e/\epsilon)$ and strictly decreasing on $(e/\epsilon,\infty)$; hence $p_u$ is the unique global maximizer of $R(\cdot,Q)$ absent the deliverability constraint. If $D(1/\epsilon)\le f$, then $p_u$ is feasible and therefore optimal. If $D(1/\epsilon)>f$, feasibility requires $D(p/e)\le f$, which (since $D$ is strictly decreasing in $p$) is equivalent to $p\ge \frac{e}{\epsilon}\ln(M/f)>\frac{e}{\epsilon}$; on this feasible set $R$ is strictly decreasing, so the unique maximizer is attained at the boundary $p=\frac{e}{\epsilon}\ln(M/f)$. Therefore $p^*(Q)$ is optimal and unique, and it is given by the theorem.
\end{proof}

\subsubsection{Proof of Proposition~\ref{prop:general_pricing_reduction}}

\begin{proof}[Proof of Proposition~\ref{prop:general_pricing_reduction}]
Fix a state $Q$. Consider the single-period problem (for this fixed $Q$) in which $(p,q)$ must satisfy $p\in\mathcal P(Q)$, $q\ge 0$, and the integrated financial feasibility constraint
\begin{equation}\label{eq:gen_fin_constraint}
C(Q)+kq\le R(p,Q),\qquad k>0,
\end{equation}
and the objective is $\max q$. Assume the optimal solution satisfies $q^*(Q)>0$.

\paragraph{Step 1: the revenue maximization problem attains a maximizer.}
Let $\tilde p\in\mathcal P(Q)$ (nonempty by assumption) and define $\tilde R:=R(\tilde p,Q)\ge 0$.
By (A2), there exists $\bar p>0$ such that for all $p\ge \bar p$,
\begin{equation}\label{eq:tail_small}
R(p,Q)\le \tilde R.
\end{equation}
Define the truncated feasible set
\[
\mathcal P_{\bar p}(Q):=\mathcal P(Q)\cap[0,\bar p].
\]
Because $\mathcal P(Q)$ is closed, $\mathcal P_{\bar p}(Q)$ is closed, and since $[0,\bar p]$ is compact, $\mathcal P_{\bar p}(Q)$ is compact. By (A1), $R(\cdot,Q)$ is continuous on $\mathcal P_{\bar p}(Q)$, hence by the Weierstrass extreme value theorem there exists $p^{\max}\in\mathcal P_{\bar p}(Q)$ such that
\[
R(p^{\max},Q)=\max_{p\in\mathcal P_{\bar p}(Q)}R(p,Q).
\]
Moreover, \eqref{eq:tail_small} implies that for any $p\in\mathcal P(Q)$ with $p\ge \bar p$,
\[
R(p,Q)\le \tilde R\le \max_{p\in\mathcal P_{\bar p}(Q)}R(p,Q)=R(p^{\max},Q),
\]
so $p^{\max}$ is also a maximizer over the full set $\mathcal P(Q)$. Therefore,
\[
p^{\max}\in\arg\max_{p\in\mathcal P(Q)}R(p,Q),
\]
and the argmax set is nonempty.

\paragraph{Step 2: in the expansion problem, any optimal price must maximize revenue.}
Let $(p^*,q^*)$ be an optimal solution of the single-period problem with $q^*>0$.
We first show the financial constraint \eqref{eq:gen_fin_constraint} binds at $(p^*,q^*)$. Suppose instead it is slack:
\[
C(Q)+kq^*<R(p^*,Q).
\]
Then there exists $\delta>0$ such that $C(Q)+k(q^*+\delta)\le R(p^*,Q)$, hence $(p^*,q^*+\delta)$ remains feasible (since $p^*\in\mathcal P(Q)$ and $q^*+\delta\ge 0$) and achieves a strictly larger objective value, contradicting optimality. Therefore,
\begin{equation}\label{eq:bind_fin_general}
C(Q)+kq^*=R(p^*,Q).
\end{equation}

Now take any $p\in\mathcal P(Q)$. The maximum feasible expansion at price $p$ is
\[
\hat q(p):=\max\Big\{q\ge 0:\ C(Q)+kq\le R(p,Q)\Big\}
=\max\left\{0,\frac{R(p,Q)-C(Q)}{k}\right\}.
\]
Because $q^*>0$, we have $R(p^*,Q)>C(Q)$ by \eqref{eq:bind_fin_general}, and hence $\hat q(p^*)=(R(p^*,Q)-C(Q))/k$.
On the region where $R(p,Q)\ge C(Q)$, the map $p\mapsto \hat q(p)$ is an increasing affine transformation of $p\mapsto R(p,Q)$. Consequently, if there existed $p'\in\mathcal P(Q)$ such that $R(p',Q)>R(p^*,Q)$, then $\hat q(p')>\hat q(p^*)=q^*$, implying the feasible pair $(p',\hat q(p'))$ achieves strictly larger objective value than $(p^*,q^*)$, a contradiction. Hence no feasible price can yield strictly higher revenue than $p^*$, i.e.
\[
p^*(Q)\in \arg\max_{p\in\mathcal P(Q)}R(p,Q).
\]
This proves the claim.
\end{proof}

\subsection{Theorem 2: Existence and uniqueness of the long-run capacity limit}

\begin{proof}
Recall that, under the integrated formulation and given the optimal pricing rule from Theorem~\ref{thm:single_period_pstar}, the optimal single-period expansion satisfies
\[
q^*(Q)=\frac{R(Q)-C(Q)}{k},
\]
where $R(Q):=R(p^*(Q),Q)$ denotes the maximized per-period VREC revenue at state $Q$. Hence $q^*(Q)=0$ if and only if
\begin{equation}\label{eq:stop_cond}
R(Q)=C(Q).
\end{equation}

By assumption, there exists $Q^\dagger$ such that the renewable deliverability constraint is non-binding for all $Q\ge Q^\dagger$. Therefore, for all $Q\ge Q^\dagger$, the optimal price is given by the unconstrained interior solution from Theorem~\ref{thm:single_period_pstar},
\[
p^*(Q)=\frac{e(Q)}{\epsilon},
\]
and thus demand is constant at
\[
D\!\left(\frac{p^*(Q)}{e(Q)}\right)=D\!\left(\frac{1}{\epsilon}\right)=M e^{-1}.
\]
Consequently, for all $Q\ge Q^\dagger$,
\begin{equation}\label{eq:R_closed_form_region}
R(Q)=p^*(Q)\,D\!\left(\frac{p^*(Q)}{e(Q)}\right)
=\frac{e(Q)}{\epsilon}\cdot M e^{-1}
=\frac{M}{e\,\epsilon}\,e(Q).
\end{equation}
Define the function
\[
F(Q):=R(Q)-C(Q),\qquad Q\ge Q^\dagger.
\]
Under \eqref{eq:R_closed_form_region}, for $Q\ge Q^\dagger$ we can write
\begin{equation}\label{eq:F_def}
F(Q)=\frac{M}{e\,\epsilon}\,e(Q)-C(Q).
\end{equation}
By the maintained regularity of primitives (in particular, continuity of $e(\cdot)$ and $C(\cdot)$), $F$ is continuous on $[Q^\dagger,\infty)$.

\smallskip
\noindent\textbf{Existence.}
The feasibility condition in the theorem is precisely
\[
C(Q^\dagger)\le \frac{M}{e\,\epsilon}\,e(Q^\dagger),
\]
which is equivalent to $F(Q^\dagger)\ge 0$. Moreover, since $C(Q)\to+\infty$ as $Q\to\infty$ and $R(Q)$ is finite for each finite $Q$ (in particular, $R(Q)=\frac{M}{e\epsilon}e(Q)$ for $Q\ge Q^\dagger$), we have
\[
\lim_{Q\to\infty}F(Q)=\lim_{Q\to\infty}\left(\frac{M}{e\,\epsilon}\,e(Q)-C(Q)\right)=-\infty.
\]
Therefore there exists $\bar Q\ge Q^\dagger$ such that $F(\bar Q)<0$. Since $F$ is continuous on $[Q^\dagger,\bar Q]$ and satisfies $F(Q^\dagger)\ge 0$ and $F(\bar Q)<0$, the Intermediate Value Theorem implies the existence of some $Q^*\in[Q^\dagger,\bar Q]$ such that $F(Q^*)=0$. By \eqref{eq:stop_cond}, this is equivalent to $q^*(Q^*)=0$.

\smallskip
\noindent\textbf{Uniqueness.}
For all $Q\ge Q^\dagger$, \eqref{eq:F_def} holds. Since $C(Q)$ is strictly increasing on $[Q^\dagger,\infty)$ and $\frac{M}{e\epsilon}e(Q)$ is fixed given $e(Q)$, it follows that $F(Q)$ is strictly decreasing on $[Q^\dagger,\infty)$ whenever $e(Q)$ is weakly decreasing on this region (which is the natural case as renewable penetration reduces grid-average emissions intensity). In that case, $F$ can cross zero at most once, so the solution $Q^*$ is unique.

Equivalently, uniqueness can be shown directly by contradiction: suppose there exist $Q_1,Q_2\ge Q^\dagger$ with $Q_1<Q_2$ such that $F(Q_1)=F(Q_2)=0$. Then
\[
\frac{M}{e\,\epsilon}\,e(Q_1)-C(Q_1)=\frac{M}{e\,\epsilon}\,e(Q_2)-C(Q_2).
\]
Rearranging yields
\[
C(Q_2)-C(Q_1)=\frac{M}{e\,\epsilon}\big(e(Q_2)-e(Q_1)\big).
\]
The left-hand side is strictly positive because $C(\cdot)$ is strictly increasing, while the right-hand side is non-positive if $e(\cdot)$ is weakly decreasing, a contradiction as $e(\cdot)$ by definition is a decreasing function. Hence there can be at most one $Q^*\ge Q^\dagger$ solving $F(Q)=0$.

\smallskip
\noindent\textbf{Characterization.}
Finally, by \eqref{eq:R_closed_form_region} and \eqref{eq:stop_cond}, any $Q^*\ge Q^\dagger$ such that $q^*(Q^*)=0$ must satisfy
\[
C(Q^*)=R(Q^*)=\frac{M}{e\,\epsilon}\,e(Q^*),
\]
which is exactly \eqref{eq:Qstar_equation}. Conversely, any $Q\ge Q^\dagger$ satisfying \eqref{eq:Qstar_equation} yields $F(Q)=0$ and therefore $q^*(Q)=0$. This establishes existence, uniqueness, and the stated characterization.
\end{proof}

\subsection{Theorem 3: Optimal Myopic Policy}

\begin{proof}
Fix any feasible policy $\mathcal P:=\{q_t^{\mathcal P},p_t^{\mathcal P}\}_{t\ge 0}$ satisfying the \emph{No overbuild} requirement
\[
0\le q_t^{\mathcal P}\le Q^\star-Q_t^{\mathcal P}\qquad \text{for all }t\ge 0,
\]
and let $\{Q_t^{\mathcal P}\}_{t\ge 0}$ denote the induced state path under the transition
\[
Q_{t+1}^{\mathcal P}=Q_t^{\mathcal P}+q_t^{\mathcal P}.
\]
Define $\bar q(Q)$ and $S(Q)=Q+\bar q(Q)$ as in \eqref{eq:monotone_reach}. By definition of $\bar q(Q)$ as the maximal \emph{feasible} one-step expansion at state $Q$, feasibility of $\mathcal P$ implies
\begin{equation}\label{eq:q_le_barq}
q_t^{\mathcal P}\le \bar q(Q_t^{\mathcal P})\qquad \text{for all }t\ge 0.
\end{equation}
Indeed, for each $t$, $(p_t^{\mathcal P},q_t^{\mathcal P})$ is a feasible pair at state $Q_t^{\mathcal P}$, and $\bar q(Q_t^{\mathcal P})$ is the supremum of feasible $q$ at that state; hence \eqref{eq:q_le_barq} holds.

Now define the myopic policy by
\[
q_t^{\mathrm{myo}}=\min\{\bar q(Q_t^{\mathrm{myo}}),\,Q^\star-Q_t^{\mathrm{myo}}\},
\qquad
Q_{t+1}^{\mathrm{myo}}=Q_t^{\mathrm{myo}}+q_t^{\mathrm{myo}}.
\]
Equivalently,
\begin{equation}\label{eq:Q_myo_update}
Q_{t+1}^{\mathrm{myo}}
=\min\{Q_t^{\mathrm{myo}}+\bar q(Q_t^{\mathrm{myo}}),\,Q^\star\}
=\min\{S(Q_t^{\mathrm{myo}}),\,Q^\star\}.
\end{equation}

We prove by induction that, for all $t\ge 0$,
\begin{equation}\label{eq:dominance_full}
Q_t^{\mathrm{myo}}\ge Q_t^{\mathcal P}.
\end{equation}

\paragraph{Base case ($t=0$).}
By assumption both trajectories start from the same initial condition, hence
\[
Q_0^{\mathrm{myo}}=Q_0^{\mathcal P}=Q_{\mathrm{init}},
\]
so \eqref{eq:dominance_full} holds for $t=0$.

\paragraph{Inductive step.}
Fix any $t\ge 0$ and assume as induction hypothesis that \eqref{eq:dominance_full} holds at time $t$, i.e.,
\[
Q_t^{\mathrm{myo}}\ge Q_t^{\mathcal P}.
\]
We show it holds at time $t+1$.

First, combining the transition equation with \eqref{eq:q_le_barq} yields
\begin{equation}\label{eq:P_next_le_S}
Q_{t+1}^{\mathcal P}
=Q_t^{\mathcal P}+q_t^{\mathcal P}
\le Q_t^{\mathcal P}+\bar q(Q_t^{\mathcal P})
=S(Q_t^{\mathcal P}).
\end{equation}
Second, by \eqref{eq:Q_myo_update},
\begin{equation}\label{eq:myo_next}
Q_{t+1}^{\mathrm{myo}}=\min\{S(Q_t^{\mathrm{myo}}),\,Q^\star\}.
\end{equation}
By the induction hypothesis $Q_t^{\mathrm{myo}}\ge Q_t^{\mathcal P}$ and the assumed \emph{Monotone reachability} (i.e.\ $S$ is non-decreasing on $[Q_{\mathrm{init}},Q^\star)$), we have
\begin{equation}\label{eq:S_mono}
S(Q_t^{\mathrm{myo}})\ge S(Q_t^{\mathcal P}).
\end{equation}
Applying the (componentwise) monotonicity of the map $x\mapsto \min\{x,Q^\star\}$ to \eqref{eq:S_mono} gives
\begin{equation}\label{eq:min_mono}
\min\{S(Q_t^{\mathrm{myo}}),\,Q^\star\}\ge \min\{S(Q_t^{\mathcal P}),\,Q^\star\}.
\end{equation}
Combining \eqref{eq:myo_next} with \eqref{eq:min_mono} yields
\begin{equation}\label{eq:myo_ge_min_S_P}
Q_{t+1}^{\mathrm{myo}}
\ge \min\{S(Q_t^{\mathcal P}),\,Q^\star\}.
\end{equation}
Finally, from \eqref{eq:P_next_le_S} and the same monotonicity of $x\mapsto \min\{x,Q^\star\}$, we obtain
\begin{equation}\label{eq:min_P_next}
\min\{S(Q_t^{\mathcal P}),\,Q^\star\}\ge \min\{Q_{t+1}^{\mathcal P},\,Q^\star\}.
\end{equation}
Putting \eqref{eq:myo_ge_min_S_P} and \eqref{eq:min_P_next} together,
\begin{equation}\label{eq:myo_ge_min_P_next}
Q_{t+1}^{\mathrm{myo}}\ge \min\{Q_{t+1}^{\mathcal P},\,Q^\star\}.
\end{equation}

At this point, we invoke \emph{No overbuild} for policy $\mathcal P$: since $0\le q_t^{\mathcal P}\le Q^\star-Q_t^{\mathcal P}$,
\[
Q_{t+1}^{\mathcal P}=Q_t^{\mathcal P}+q_t^{\mathcal P}\le Q_t^{\mathcal P}+(Q^\star-Q_t^{\mathcal P})=Q^\star,
\]
hence $\min\{Q_{t+1}^{\mathcal P},Q^\star\}=Q_{t+1}^{\mathcal P}$. Substituting into \eqref{eq:myo_ge_min_P_next} yields
\[
Q_{t+1}^{\mathrm{myo}}\ge Q_{t+1}^{\mathcal P},
\]
which is exactly \eqref{eq:dominance_full} at time $t+1$. This completes the induction.

Therefore, \eqref{eq:dominance_full} holds for all $t\ge 0$.

\smallskip
\noindent\textbf{Global optimality.}
The per-period stage cost is $-q_t$, so minimizing total discounted (or undiscounted) stage cost is equivalent to maximizing cumulative expansion $\sum_t q_t$ until the process reaches $Q^\star$. Since $Q_{t+1}=Q_t+q_t$, the dominance result \eqref{eq:dominance_full} implies that the myopic policy reaches any given intermediate capacity level weakly earlier than any feasible policy and, in particular, reaches $Q^\star$ in weakly fewer periods whenever $Q^\star$ is reachable. Hence the myopic policy yields weakly larger cumulative expansion at every finite horizon and therefore is globally optimal among feasible policies.
\end{proof}

\subsection{Propositions with Optimal Revenue Sharing}

\begin{proof}[Proof of Proposition~\ref{prop:gamma_star}]
Fix a state $Q$ and abbreviate $e:=e(Q)$, $f:=f(Q)$, $C_1(q):=C_1(Q,q)=C_S(Q)+kq$, $C_2:=C_2(Q)$, and $R(p):=R(p,Q)$.
Rewrite \eqref{eq:opt_problem_single_rs} in standard $\le 0$ form:
\[
\begin{aligned}
h_1(p,q,\gamma)&:=D\!\left(\frac{p}{e}\right)-f \le 0,\\
h_2(p,q,\gamma)&:=C_1(q)-(1-\gamma)R(p)\le 0,\\
h_3(p,q,\gamma)&:=C_2-\gamma R(p)\le 0,\\
h_4(p,q,\gamma)&:=-q\le 0,\qquad
h_5(p,q,\gamma):=-p\le 0,\\
h_6(p,q,\gamma)&:=-\gamma\le 0,\qquad
h_7(p,q,\gamma):=\gamma-1\le 0.
\end{aligned}
\]
Let multipliers $(\lambda,\mu,\theta,\nu,\eta,\alpha,\beta)\ge 0$ correspond to $(h_1,\dots,h_7)$, respectively. The Lagrangian is
\[
\mathcal{L}(p,q,\gamma)
= q-\lambda h_1-\mu h_2-\theta h_3-\nu h_4-\eta h_5-\alpha h_6-\beta h_7.
\]
Using $D'(x)=-\epsilon D(x)$ and $R_p(p)=D(p/e)\big(1-\epsilon p/e\big)$, stationarity gives:
\begin{align}
\frac{\partial \mathcal{L}}{\partial q}=0
&\iff 1-\mu k+\nu=0, \label{eq:rs_kkt_q}\\
\frac{\partial \mathcal{L}}{\partial p}=0
&\iff \lambda\frac{\epsilon}{e}D\!\left(\frac{p}{e}\right)
+(\mu(1-\gamma)+\theta\gamma)\,R_p(p)+\eta=0, \label{eq:rs_kkt_p}\\
\frac{\partial \mathcal{L}}{\partial \gamma}=0
&\iff -\mu R(p)+\theta R(p)+\alpha-\beta=0. \label{eq:rs_kkt_gamma}
\end{align}
Primal feasibility, dual feasibility, and complementary slackness are:
\begin{align}
&D(p/e)\le f,\quad C_1(q)\le (1-\gamma)R(p),\quad C_2\le \gamma R(p),\quad q,p\ge 0,\quad \gamma\in[0,1], \label{eq:rs_pf}\\
&\lambda,\mu,\theta,\nu,\eta,\alpha,\beta\ge 0, \label{eq:rs_df}\\
&\lambda\!\left(D(p/e)-f\right)=0,\quad
\mu\!\left(C_1(q)-(1-\gamma)R(p)\right)=0,\quad
\theta\!\left(C_2-\gamma R(p)\right)=0,\label{eq:rs_cs1}\\
&\nu q=0,\quad \eta p=0,\quad \alpha \gamma=0,\quad \beta(\gamma-1)=0. \label{eq:rs_cs2}
\end{align}

Assume the optimal solution exhibits strictly positive expansion $q^*(Q)>0$. Then $\nu^*=0$ by \eqref{eq:rs_cs2}, and \eqref{eq:rs_kkt_q} implies
\begin{equation}\label{eq:mu_star_rs}
\mu^*=\frac{1}{k}>0.
\end{equation}
Moreover, as in the integrated benchmark, if $C_1(q^*)<(1-\gamma^*)R(p^*)$, then one can increase $q$ slightly while keeping $(p^*,\gamma^*)$ fixed and all constraints feasible, contradicting optimality. Hence the operator account constraint binds:
\begin{equation}\label{eq:op_bind}
C_1(q^*)=(1-\gamma^*)R(p^*).
\end{equation}

Now consider the choice of $\gamma$.

\smallskip
\noindent\textbf{Case 1: $C_2\le 0$.}
Then for any $(p,q)$ with $R(p)\ge 0$ and any $\gamma\ge 0$,
\[
C_2\le 0\le \gamma R(p),
\]
so the generator-viability constraint $h_3\le 0$ is slack at $\gamma=0$. Thus $\theta^*=0$ by complementary slackness \eqref{eq:rs_cs1}. Because $q^*>0$ forces $R(p^*)>0$ (otherwise \eqref{eq:op_bind} cannot hold with $C_1(q^*)>0$), the slackness of $h_3$ at $\gamma=0$ implies that setting $\gamma=0$ is feasible for generators.

We now show $\gamma^*=0$ is optimal. Take any feasible solution with $\gamma>0$ and the same $(p,q)$. Decreasing $\gamma$ weakly increases $(1-\gamma)R(p)$ and therefore relaxes \eqref{eq:opt_problem_single_rs_b} while leaving \eqref{eq:opt_problem_single_rs_a} unchanged; since $C_2\le 0$, \eqref{eq:opt_problem_single_rs_c} remains satisfied at $\gamma=0$. Hence the feasible set in $(p,q)$ weakly expands as $\gamma$ decreases, and the objective $\max q$ cannot be improved by choosing $\gamma>0$. Therefore $\gamma^*(Q)=0$, consistent with \eqref{eq:gamma_star}.

\smallskip
\noindent\textbf{Case 2: $C_2>0$ and $\gamma^*(Q)\in(0,1)$.}
If $\gamma^*\in(0,1)$, then the bound constraints on $\gamma$ are slack, hence $\alpha^*=\beta^*=0$ by \eqref{eq:rs_cs2}. Stationarity in $\gamma$, \eqref{eq:rs_kkt_gamma}, gives
\[
(-\mu^*+\theta^*)R(p^*)=0.
\]
Since $q^*>0$ and \eqref{eq:op_bind} imply $R(p^*)>0$, we conclude
\begin{equation}\label{eq:theta_eq_mu}
\theta^*=\mu^*=\frac{1}{k}>0.
\end{equation}
Because $\theta^*>0$, complementary slackness in \eqref{eq:rs_cs1} implies the generator-viability constraint binds:
\[
C_2=\gamma^* R(p^*).
\]
Equivalently,
\[
\gamma^*=\frac{C_2}{R(p^*)}.
\]
Finally, evaluating at the expansion-maximizing price in the integrated benchmark $p^*(Q)$ and letting $R^*(Q):=R(p^*(Q),Q)$ yields
\[
\gamma^*(Q)=\frac{C_2(Q)}{R^*(Q)}.
\]
Combining with the previous case $C_2(Q)\le 0\Rightarrow \gamma^*(Q)=0$ proves \eqref{eq:gamma_star}.
\end{proof}

\begin{proof}[Proof of Proposition~\ref{prop:interior_equivalence}]
Fix $Q$ and suppose the optimal revenue-sharing choice is interior, $\gamma^*(Q)\in(0,1)$, and the optimal expansion is strictly positive, $q^*(Q)>0$.
Using the KKT system developed above, interiority implies $\alpha^*=\beta^*=0$ (slack bounds on $\gamma$), and $q^*>0$ implies $\nu^*=0$.

\smallskip
\noindent\textbf{Step 1: Both separated financial constraints bind.}
From $q^*(Q)>0$, the same argument as in \eqref{eq:op_bind} yields that the operator account constraint binds:
\[
C_1(Q,q^*)=(1-\gamma^*)R(p^*,Q),
\]
and hence $\mu^*>0$.
From the $\gamma$-stationarity \eqref{eq:rs_kkt_gamma} with $\alpha^*=\beta^*=0$ we obtain $(\theta^*-\mu^*)R(p^*,Q)=0$.
Because $q^*>0$ implies $R(p^*,Q)>0$, it follows that $\theta^*=\mu^*>0$, so $\theta^*>0$.
By complementary slackness for $h_3$ in \eqref{eq:rs_cs1}, $\theta^*>0$ implies the generator-viability constraint binds:
\[
C_2(Q)=\gamma^* R(p^*,Q).
\]
Therefore, at an interior optimum,
\begin{equation}\label{eq:both_bind}
C_1(Q,q^*)=(1-\gamma^*)R(p^*,Q),
\qquad
C_2(Q)=\gamma^*R(p^*,Q).
\end{equation}

\smallskip
\noindent\textbf{Step 2: Aggregation yields the integrated financial constraint.}
Adding the two equalities in \eqref{eq:both_bind} gives
\[
C_1(Q,q^*)+C_2(Q)=R(p^*,Q).
\]
Recalling $C_1(Q,q)=C_S(Q)+kq$ and $C_2(Q)=C_R(Q)-f(Q)\pi(Q)$, the left-hand side is exactly
\[
C_S(Q)+C_R(Q)-f(Q)\pi(Q)+kq^*
= C(Q)+kq^*,
\]
so we obtain the integrated financial feasibility constraint
\begin{equation}\label{eq:agg_integrated}
C(Q)+kq^* = R(p^*,Q).
\end{equation}

\smallskip
\noindent\textbf{Step 3: The pricing problem reduces to the integrated benchmark.}
The deliverability constraint \eqref{eq:opt_problem_single_rs_a} is identical to that in the integrated formulation. Under $q^*>0$, \eqref{eq:agg_integrated} implies that maximizing $q$ is equivalent to maximizing $R(p,Q)$ over deliverable $p$, exactly as in the integrated benchmark. Hence the optimal price under revenue sharing coincides with the integrated optimal price:
\[
p^{\mathrm{RS}}(Q)=p^*(Q).
\]
Substituting $p^{\mathrm{RS}}(Q)=p^*(Q)$ into \eqref{eq:agg_integrated} yields the same optimal expansion as in the integrated benchmark:
\[
q^{\mathrm{RS}}(Q)=\frac{R(p^{\mathrm{RS}}(Q),Q)-C(Q)}{k}
=\frac{R(p^*(Q),Q)-C(Q)}{k}
=q^*(Q).
\]

\smallskip
\noindent\textbf{Step 4: Multi-period implications.}
Whenever $\gamma^*(Q)\in(0,1)$ holds along the realized capacity path, the per-period feasible set in $(p,q)$ under \eqref{eq:opt_problem_single_rs} coincides with that of the integrated benchmark because \eqref{eq:agg_integrated} is equivalent to the single integrated financial constraint and deliverability is unchanged. Therefore the induced one-step expansion map $\bar q(Q)$ and reachability map $S(Q)$ are identical under the two formulations along that path. All multi-period results that depend only on these per-period objects (including Theorems~\ref{thm:single_period_pstar}, \ref{thm:long_run_capacity}, and \ref{thm:myopic_clean}) carry over.
\end{proof}

\section{Appendix B: Simulation Setup}

\subsection{ISO-NE test system}

Our numerical study is conducted on a modified ISO New England (ISO-NE) test system. In general, ISO-NE is the regional transmission organization responsible for operating the electric grid and wholesale electricity markets across the six New England states. It provides a useful benchmark for simulation because it represents a transmission-constrained, multi-zone power system with geographically differentiated generation, demand, and congestion patterns. In market-design studies, ISO-NE-style test systems are often used to capture the interaction between network constraints, nodal or zonal prices, generator dispatch, and the integration of renewable and storage resources.

The specific environment adopted here follows the modified 8-zone ISO-NE test system used in \citet{qi_locational_2025}. In that setup, the system contains 8 nodes, 12 transmission lines, and 76 thermal generators, with a total installed conventional capacity of 23.1\,GW and an average load of 13\,GW. As emphasized by \citet{qi_locational_2025}, this is an agent-based market simulation environment with multiple renewable and storage units, intended as a transferable testbed rather than a result that depends uniquely on ISO-NE.

From a modeling perspective, the network is represented through a DC optimal power flow and economic dispatch structure. The dispatch includes system-wide power balance, transmission flow limits, generator operating limits, reserve requirements, ramping constraints, and storage charging/discharging and state-of-charge dynamics. In the notation of \citet{qi_locational_2025}, the core system consists of sets of conventional generators, storages, time periods, nodes, and lines, with nodal prices defined through the locational marginal price (LMP) expression induced by the dual variables of the dispatch problem.

\subsection{Baseline market-dispatch framework}

The simulation environment follows the market-dispatch framework developed for studying storage bidding and bid bounds under uncertainty. The framework contains three related dispatch layers.

First, it define an Oracle Economic Dispatch (OED) benchmark, which assumes perfect foresight of demand and renewable profiles over a multi-period horizon. This benchmark is not intended to be implementable in practice; rather, it serves as the ideal reference against which storage opportunity costs and truthful bids are conceptually defined.

Second, it introduce a Single-Period Economic Dispatch (SED) formulation for real-time operation. This is the operational dispatch used when storage submits charge and discharge bids in each time step. In this setting, the previous-period state of charge is treated as given, and the market clears storage and generator dispatch subject to power balance, network limits, and operating constraints.

Third, to model uncertainty explicitly, it develop a Chance-Constrained Economic Dispatch (CED). This reformulation incorporates net-load uncertainty through chance constraints, allowing the system operator to choose a confidence level and thereby internalize both uncertainty and risk preference into the resulting storage bid bounds. Their main theoretical result is that the storage opportunity value derived from the chance-constrained formulation provides a probabilistic upper bound on truthful storage marginal costs.

\begin{figure}[htbp]
    \centering
    
    \includegraphics[width=0.48\textwidth]{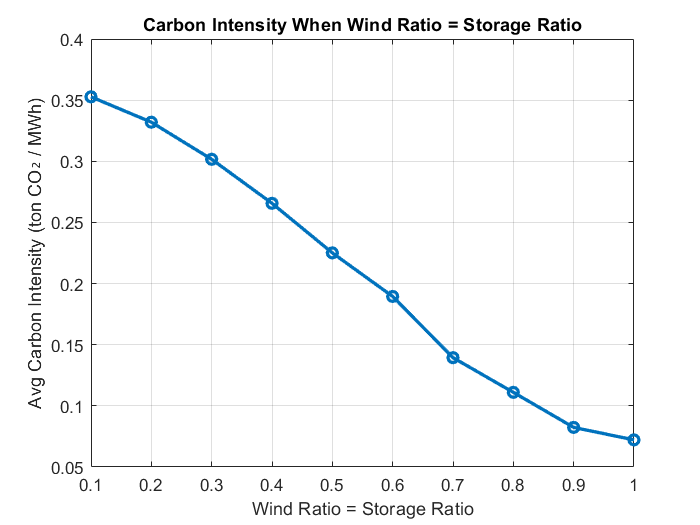}
    \hfill
    \includegraphics[width=0.48\textwidth]{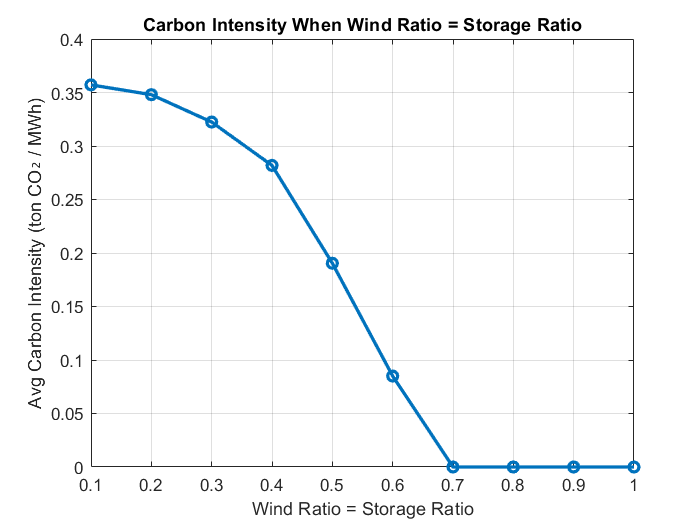}

    \vspace{0.5em}

    \includegraphics[width=0.48\textwidth]{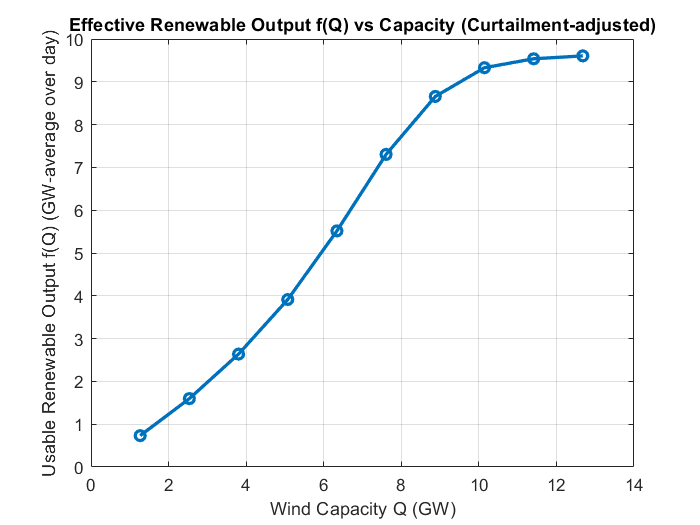}
    \hfill
    \includegraphics[width=0.48\textwidth]{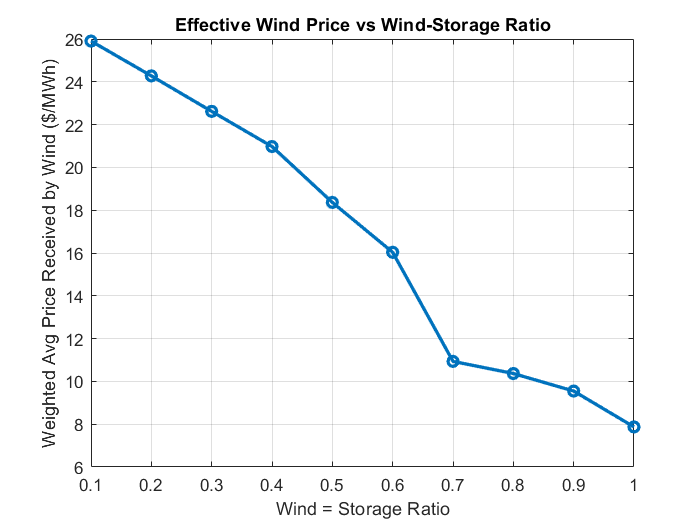}

    \caption{Grid-property visualization for the ISO-NE test system. Subfigure 1 shows the carbon intensity of the grid when operational prediction error and uncertainty are incorporated, while Subfigure 2 shows the corresponding relationship under perfect prediction for comparison. Subfigure 3 reports effective renewable output as a function of installed wind capacity, accounting for curtailment, and exhibits the broadly concave pattern assumed in the model. Subfigure 4 shows the wholesale market price received by renewable generation (wind) as wind capacity increases. This relationship is monotonically decreasing but neither globally convex nor concave, highlighting the nonlinear effects of uncertainty and market-clearing dynamics in the power system.}
    \label{fig:grid_properties_placeholder}
\end{figure}

\subsection{Grid Properties and Simulation Settings}

Figure~\ref{fig:grid_properties_placeholder} presents four summary plots, summarizing the key empirical relationships in the ISO-NE simulation environment that inform the functional structure assumed in our model.

\begin{table}[htbp]
\centering
\caption{Baseline parameters and cost-function specification.}
\label{tab:baseline_parameters}
\begin{tabularx}{\textwidth}{l l X}
\hline
\textbf{Parameter} & \textbf{Unit} & \textbf{Remarks} \\
\hline
$Q_0 = 0.5$ & GW & Initial renewable capacity in the baseline numerical simulation; sample initial condition. For reference, the current ISO-NE stage is approximately $5.0$ GW in the corresponding scaled interpretation. \\

$M = 10$ & GW & Market size, interpreted as the peak demand of the modeled utility. \\

$\epsilon = 0.0045$ & N/A & Demand sensitivity parameter used in the voluntary procurement specification. \\

$k = 1000$ & M\$/GW (or \$/kW) & Renewable installation cost parameter. Since $1000$ M\$/GW = $1000$ \$/kW, the two units are equivalent. \\

$\text{wind\_cf} = 0.35$ & N/A & Wind capacity factor 35\% used to convert installed wind capacity into effective energy production. \\

$\alpha_R = 0.015 \times 1400 = 21$ & M\$/GW$\cdot$yr & Baseline annual O\&M cost coefficient in the renewable cost function $c_R(Q)$. 15\% of previous renewable capacity, which has \$1400/kW, higher than future expansions. \\

$\beta_R = 5$ & M\$/GW$^2\cdot$yr & Convexity coefficient in the renewable cost function, capturing site-quality depletion as renewable capacity expands. \\

$c_R(Q) = \alpha_R Q + \beta_R Q^2$ & M\$/yr & Renewable operating cost function. Includes annual operating cost only, with increasing marginal cost as $Q$ rises. \\

$\alpha_S = 0.08 \times 120 = 9.6$ & M\$/GW$\cdot$yr & Effective annualized system cost coefficient in $c_S(Q)$, motivated by storage scaling one-to-one with wind capacity and assuming an 8\% annual payment on a battery cost of 120 M\$/GWh with 4-hour duration. \\

$\beta_S = 1$ & M\$/GW$^2\cdot$yr & Convexity coefficient in the system cost function, representing gradually harder system integration at higher renewable penetration. \\

$c_S(Q) = \alpha_S Q + \beta_S Q^2$ & M\$/yr & System cost function. Captures the effective annual cost of supporting renewable integration, including storage-related system costs. \\
\hline
\end{tabularx}
\end{table}

\end{document}